\documentclass[12pt]{article}
\usepackage{amsthm}
\usepackage{varioref}
\usepackage{amsmath,amstext,amsthm,a4,amssymb,amscd}   
\usepackage[T1]{fontenc}
\usepackage[utf8]{inputenc}
\usepackage{lmodern}
\usepackage{array}
\usepackage{latexsym}
\usepackage{fancyhdr}
\usepackage{mathrsfs}\let\cal\mathscr
\usepackage[all]{xy}
\usepackage{makeidx}   
\usepackage{color}
\usepackage{pifont}
\usepackage{amsfonts,amssymb}
\usepackage{hyperref}
\usepackage{cleveref}
\usepackage[numbers,square]{natbib}
\usepackage[disable]{todonotes}
\usepackage{verbatim}
\usepackage{relsize}

\usepackage{avant}
\usepackage[T1]{fontenc}      

\usepackage{amsfonts,amssymb}  
 
\usepackage{graphicx}

\usepackage{natbib}
\newcommand \Om {\Omega}
\newcommand \om {\omega}
\newcommand \0 {\emptyset}

\renewcommand \leq {\leqslant}
\renewcommand \geq {\geqslant}

\DeclareMathOperator{\Vol}{Vol}
\DeclareMathOperator{\End}{End}

\DeclareMathOperator{\Tr}{Tr}

\DeclareMathOperator{\Ker}{Ker}

\DeclareMathOperator{\GL}{GL}

\DeclareMathOperator{\Ric}{Ric}

\DeclareMathOperator{\Aut}{Aut}
\DeclareMathOperator{\Herm}{Herm}
\DeclareMathOperator{\scal}{scal}
\DeclareMathOperator{\Prod}{Prod}
\DeclareMathOperator{\Met}{Met}
\DeclareMathOperator{\Hilb}{Hilb}
\DeclareMathOperator{\Kod}{Kod}
\DeclareMathOperator{\FS}{FS}
\DeclareMathOperator{\ev}{ev}

\newcommand \dbar {\overline{\partial}}

\newcommand \< {\mathcal{h}}
\renewcommand \> {\mathcal{i}}

\newcommand \cinf {\CC^\infty}

\newcommand \Id {{\rm Id}}

\renewcommand \epsilon {\varepsilon}

\newcommand \CC {{\cal C}}
\newcommand \BB {{\cal B}}

\newcommand \EE {{\cal E}}

\newcommand \HH {{\cal H}}

\newcommand \MM {{\cal M}}

\newcommand \TT {{\cal T}}

\newcommand \s {\textbf{s}}

\def\cL{\mathscr{L}}

\def\cL{\mathscr{L}}

\newcommand{\til}[1]{\widetilde{#1}}

\newcommand \dt {\frac{\partial}{\partial t}}

\newcommand \R {\mathbb R}
\newcommand \IP {\mathbb P}

\newcommand \C {\mathbb C}

\newcommand \N {\mathbb N}

\newcommand \n {\mathbb N}

\newcommand \fl {\rightarrow}

\newcommand \ignore[1] {}

\theoremstyle{plain}
\newtheorem{theorem}{Theorem}[section]
\newtheorem{lem}[theorem]{Lemma}
\newtheorem{cor}[theorem]{Corollary}
\newtheorem{prop}[theorem]{Proposition}
\theoremstyle{definition}
\newtheorem{defi}[theorem]{Definition}
\newtheorem{ex}[theorem]{Example}

\newtheorem{rmk}[theorem]{Remark}

\numberwithin{equation}{section}

\crefname{equation}{}{}
\crefname{lem}{Lemma}{Lemmas}
\crefname{theorem}{Theorem}{Theorems}
\crefname{cor}{Corollary}{Corollaries}
\crefname{ex}{Example}{Examples}
\crefname{defi}{Definition}{Definitions}
\crefname{prop}{Proposition}{Propositions}
\crefname{section}{Section}{Sections}
\crefname{subsection}{Section}{Sections}
\crefname{rmk}{Remark}{Remarks}
\crefname{nota}{Notation}{Notations}

\hypersetup{
    colorlinks=true,
    linkcolor=red,
    filecolor=magenta,      
    urlcolor=cyan,
}

\begin{document}

\title{\bf{Anticanonically balanced metrics on Fano manifolds}}
\author{Louis IOOS$^1$}
\date{}
\maketitle
\newcommand{\Addresses}{{
  \bigskip
  \footnotesize

  \textsc{Philipps-Universität Marburg, Hans-Meerwein-Strasse 6, 35043 Marburg,
Germany}\par\nopagebreak
  \textit{E-mail address}: \texttt{ioos@mathematik.uni-marburg.de}\par\nopagebreak
  \textit{Website}: \texttt{louisioos.github.io}
}}

\footnotetext[1]{Partially supported by the European Research Council Starting grant 757585}

\begin{abstract}
We show that if a Fano manifold has discrete automorphism
group and admits a polarized Kähler-Einstein metric,
then there exists
a sequence of anticanonically balanced metrics converging
smoothly to the Kähler-Einstein metric.
Our proof is based on a simplification of
Donaldson's proof of the analogous result for balanced metrics,
replacing a delicate geometric
argument by the use of Berezin-Toeplitz quantization.
We then apply this result to compute the asymptotics of
the optimal rate of convergence to the fixed point
of Donaldson's iterations in the anticanonical setting.
\end{abstract}

\section{Introduction}

A fundamental question in the study of a
compact complex manifold $X$ is the existence
of a \emph{canonical Riemannian metric},
which reflects its complex geometry in the best possible way.
When $X$ comes endowed with an ample holomorphic line
bundle $L$, one should look for such metrics
inside the set of \emph{polarized Kähler metrics}
induced by positive Hermitian metrics on $L$.
In case $X$ is a \emph{Fano manifold}, so that its
\emph{anticanonical line bundle} $K_X^*:=\det(T^{(1,0)}X)$ is
ample, the ideal candidate for such a canonical Riemannian metric
is a polarized \emph{Kähler-Einstein} metric.
By a result of Bando and Mabuchi in \cite{BM87},
if such
a Kähler-Einstein metric exists, then it is unique.
However, finding
Kähler-Einstein metrics on Fano manifolds
is an extremely difficult problem,
and existence is related to deep
properties of $X$ as a complex algebraic manifold
\cite{CDS15,Tia15}.

A fruitful approach in finding a Kähler-Einstein metric
on $X$, when it exists, is to approximate it by yet
another type of canonical
metrics, the so-called \emph{anticanonically balanced metrics},
which are associated with a natural sequence of
projective embeddings of $X$. To define them,
first recall that a holomorphic line bundle $L$
over a compact complex manifold $X$ is \emph{ample}
if it admits a \emph{positive Hermitian metric}
$h\in\Met^+(L)$, so that its \emph{Chern curvature}
$R_h\in\Om^2(X,\C)$ induces a \emph{Kähler form}
on $X$ via the formula
\begin{equation}\label{preq}
\om_h:=\frac{\sqrt{-1}}{2\pi}R_h\,.
\end{equation}
Writing $J\in\End(TX)$ for the complex structure of $X$,
this means that the following formula
defines a Riemannian metric on $X$, called a
\emph{polarized Kähler metric},
\begin{equation}\label{gTXintro}
g^{TX}_h:=\om_h(\cdot,J\cdot)\,.
\end{equation}
Assume now that $X$ is a Fano manifold,
so that $L:=K_X^*$ is ample, and fix $p\in\N$ big enough.
Consider
the \emph{Kodaira embedding}
of $X$ into the projective space
of hyperplanes in the space $H^0(X,L^p)$ of holomorphic
sections of the tensor power $L^p:=L^{\otimes p}$.
Via this embedding, $L^p$ is identified
with the restriction of the dual tautological
line bundle, and given a Hermitian product
$H\in\Prod(H^0(X,L^p))$ on $H^0(X,L^p)$,
one gets an induced positive Hermitian metric
$\FS(H)\in\Met^+(L^p)$ on $L^p$, called
\emph{Fubini-Study metric}.
Conversely, given a positive Hermitian metric
$h^p\in\Met^+(L^p)$,
one can consider the Hermitian
inner product $\Hilb_\nu(h^p)\in\Prod(H^0(X,L^p))$
defined on $s_1,\,s_2\in H^0(X,L^p)$ by
\begin{equation}\label{L2intro}
\<s_1,s_2\>_{\Hilb_\nu(h^p)}:=\frac{n_p}{\Vol(d\nu_{h})}\int_X\,
\<s_1(x),s_2(x)\>_{h^p}\,d\nu_{h}(x)\,,
\end{equation}
where $d\nu_h$ is the \emph{anticanonical volume form}
induced by $h\in\Met^+(L)$,
defined over any contractible open subset
$U\subset X$ by the formula
\begin{equation}\label{dnucandef}
d\nu_h:=\sqrt{-1}^{\,n^2}\frac{\theta\wedge\bar{\theta}}
{~~~|\theta|_{h^{-1}}^2}\,,
\end{equation}
for any non-vanishing $\theta\in\cinf(U,K_X)$, where
$h^{-1}$ denotes the Hermitian metric on $K_X$ induced by
$h\in\Met^+(K_X^*)$.
A Hermitian metric $h_p\in\Met^+(L^p)$ is called
\emph{anticanonically balanced} if it coincides
with the Fubini-Study metric induced by the
\emph{Hilbert product} \cref{L2intro}, i.e., if
\begin{equation}
h_p=\FS(\Hilb_\nu(h_p))\,.
\end{equation}
These metrics have been introduced by Donaldson in \cite{Don09}.
Note that the original concept of a
\emph{balanced metric}, introduced by Donaldson in \cite{Don01}
and which we describe in
\cref{Liouvilleex},
uses the \emph{Liouville volume form} $\om_h^n/n!$ in
the Hilbert product
\cref{L2intro} instead of the
anticanonical volume form \cref{dnucandef}.
By a result of Berman, Boucksom, Guedj and
Zeriahi in \cite[\S\,7]{BBGZ13},
if an anticanonically balanced metric $h_p\in\Met^+(L^p)$ exists,
then it is unique up to a multiplicative
constant in $\Met^+(L^p)$.
On the other hand, a polarized Kähler-Einstein metric
is characterized by the property that the associated
anticanonical
volume form \cref{dnucandef} coincides with the
associated
Liouville volume form up to a multiplicative
constant.

In \cref{anticansec}, we present a new proof of the following Theorem.
For any $m\in\N$, let $|\cdot|_{\CC^m}$ be a fixed
$\CC^m$-norm on $\Om^2(X,\R)$.

\begin{theorem}\label{mainth}
Let $X$ be a Fano manifold with discrete automorphism group
admitting a polarized Kähler-Einstein metric, and write
$L:=K_X^*$.
Then for any $m\in\N$, there exists $C_m>0$
and a sequence of positive Hermitian metrics
$\{h_p\in\Met^+(L^p)\}_{p\in\N}$,
which are anticanonically
balanced for all $p\in\N$ big enough and
such that
\begin{equation}\label{mainthfla}
\left|\,\frac{1}{p}\om_{h_p}-\om_\infty\,\right|_{\CC^m}
\leq \frac{C_m}{p}\,,
\end{equation}
where $\om_\infty\in\Om^2(X,\R)$
is the Kähler form associated with the polarized
Kähler-Einstein metric.
\end{theorem}

This result
has first been announced by Keller in \cite[Th.\,5]{Kel09}.
A proof of existence
and weak convergence
in the sense of currents has first been given by
Berman, Boucksom, Guedj and
Zeriahi in \cite[Th.\,7.1]{BBGZ13}, and a proof of smooth convergence
has then been given by Takahashi in \cite[Th. 1.3]{Tak21},
extending the original
proof of Donaldson in \cite{Don01} of the analogous result for the Liouville
volume form.

Our proof of \cref{mainth} also follows the basic strategy of Donaldson's
proof,
constructing approximately balanced metrics
using the asymptotic expansion of the \emph{Bergman kernel}
along the diagonal \cite{Cat99,Lu00,Tia90,Zel98}
and showing the convergence of the gradient flow of the norm squared
of the associated \emph{moment map} close to a zero.
However, the most technical part of Donaldson's proof,
which consists in estimating
the derivative of the moment map from below, has no straightforward analogue
in the anticanonical case. In fact, in the original case of Donaldson,
the derivative of the moment
map has a geometric interpretation, which has been clarified by
Phong and Sturm in \cite{PS04}, giving a natural lower bound.
By contrast, in the anticanonical
case of \cref{mainth}, there is no obvious geometric interpretations
for the derivative of the moment map,
and adapting \cite[Th.\,2]{PS04} is a serious
difficulty, which was only overcome recently by Takahashi in
\cite[Prop. 3.5]{Tak21}.
The main novelty of our method is to replace
this geometric
argument by the use of the asymptotics of the
\emph{spectral gap of the Berezin transform}
established in \cite[Th.\,3.1]{IKPS19}.
More precisely, we use the equivalent
asymptotics for the spectral gap of the
\emph{Berezin-Toeplitz quantum channel}, recalled in \cref{Bpgap},
which can be understood as 
the operation of dequantization followed by quantization
of a quantum observable, i.e.,
the \emph{Berezin-Toeplitz quantization of its Berezin symbol}.
This strategy was inspired by the work of Fine in \cite{Fin12},
who studied the derivative of the moment map in the original setting of
Donaldson, assuming the existence of a balanced metric.

In \cref{donsec}, we use \cref{mainth} together
with the techniques of \cite{IKPS19} and the energy functional
of \cite[\S\,7]{BBGZ13}
to establish the exponential convergence of \emph{Donaldson's
iterations} towards the anticanonically balanced metric
for each $p\in\N$ big enough,
and compute the asymptotics of the optimal rate of convergence
as $p\to+\infty$.
To explain this result, let us fix $p\in\N$ big enough,
and define the anticanonical \emph{Donaldson map}
on the space $\Prod(H^0(X,L^p))$ of Hermitian inner products on
$H^0(X,L^p)$ by
\begin{equation}
\TT_\nu:=\Hilb_\nu\circ\,\FS:\Prod(H^0(X,L^p))\longrightarrow
\Prod(H^0(X,L^p))\,.
\end{equation}
A fixed point $H\in\Prod(H^0(X,L^p))$
of this map is called an \emph{anticanonically balanced product}.
It has been introduced by Donaldson in \cite{Don05,Don09}
for various different
volume forms in the Hilbert product \cref{L2intro},
and has been used as a dynamical system approximating
the corresponding balanced metric, seen as the
Fubini-Study metric $\FS(H)\in\Met^+(L^p)$
associated with a fixed point.
Our main result in this context is the following,
where we use the natural distance on
$\Prod(H^0(X,L^p))$ as a symmetric space.

\begin{theorem}\label{gapth}
Let $X$ be a Fano manifold with discrete automorphism group
and admitting a polarized Kähler-Einstein metric.
Then for any 
$p\in\N$ big enough, there exists $\beta_p\in\,]0,1[$
such that for any $H_0\in\Prod(H^0(X,L^p))$,
there exists an anticanonically
balanced product $H\in\Prod(H^0(X,L^p))$
and a constant $C>0$ such that for all
$k\in\N$, we have
\begin{equation}\label{expcvest}
\textup{dist}\left(\TT_\nu^k(H_0)\,,\,H\right)\leq
C\beta^k_p\;.
\end{equation}
Furthermore, the constant $\beta_p\in\,]0,1[$
satisfies the following estimate as $p\fl+\infty$,
\begin{equation}\label{betap}
\beta_p=1-\frac{\lambda_1-4\pi}
{4\pi p}+ O(p^{-2})\,,
\end{equation}
where $\lambda_1>4\pi$ is the first positive eigenvalue of the
Riemannian Laplacian associated with the
polarized Kähler-Einstein metric acting on  $\cinf(X,\C)$,
and this estimate is sharp.
\end{theorem}

This extends the results of
\cite[Th.\,4.4, Rmk.\,4.12]{IKPS19} to the anticanonical
setting.
As explained in \cref{gaprmk}, this confirms a prediction of
Donaldson in \cite{Don09} on the compared rates of convergence of the iterations
associated to various notions of balanced products.
Note that the \emph{smooth} convergence
of the Kähler forms in \cref{mainth} is necessary to compute the
rate of convergence \cref{betap}. On the other hand,
the proof of simple convergence in \cref{gapth} follows
from the work of Berman in \cite[Prop.\,2.9]{Ber13},
and is based on the
convexity of an appropriate
energy functional, which has been established in
\cite[Lemma\,7.2]{BBGZ13}
based on the results of Berndtsson in \cite{Ber09a,Ber09b}
on the positivity of direct images.
Note that the exponential convergence of the iterations
follows from the estimate \cref{betap} thanks to the
strict lower bound $\lambda_1>4\pi$ on the first
positive eigenvalue of the Kähler-Einstein Laplacian,
which holds under the necessary assumption
of discrete automorphism group as a consequence
of a classical result
of Lichnerowicz \cite{Lic59} and Matsushima \cite{Mat57}.
This lower bound plays a fundamental role in the proofs of both
\cref{mainth,gapth}, in particular in \cref{approxbal}
to construct approximately
balanced metrics and in \cref{dmu} via the asymptotics
of the spectral gap of the quantum channel.
\cref{gapth} also complements the work of Liu and Ma in \cite{LM07}, who
established the convergence of the refined approximations
of Donaldson in \cite[\S\,2.2.1]{Don09}.

The advantage of our proof of \cref{mainth} is
that it can be adapted in a systematic way to various
choices of a volume form in the Hilbert product
\cref{L2intro}, leading to the various notions
of balanced metrics.
In \cref{setting}, we give the general set-up for an
arbitrary \emph{volume map} \cref{nuhdef} on the space
$\Met^+(L)$ of positive
Hermitian metrics on an ample holomorphic line bundle $L$
over a compact complex manifold $X$.
This includes in particular
the \emph{$\nu$-balanced metrics} on
Calabi-Yau
manifolds and the
\emph{canonically balanced metrics} on manifolds
with ample canonical line bundle,
introduced by Donaldson in \cite{Don09}
and which we describe in \cref{cstantex,canex}.
The proof given in \cref{anticansec} can readily be
adapted to
these two cases, which do not need any assumption on the
automorphism group and are in fact easier.
We present the proof in the case of Fano manifolds only
because it is the most delicate one, as the Kähler-Einstein metric does not
exists a priori.
The smooth convergence of
$\nu$-balanced metrics to the polarized \emph{Yau metric}
associated to $d\nu$
has been outlined by Donaldson in
\cite[\S\,2.2]{Don09},
and then established by Keller in \cite[Th.\,4.2]{Kel09}
as a consequence of a result of Wang in \cite{Wan05}.
The differential of the associated moment map
at a $\nu$-balanced embedding has been studied by
Keller, Meyer and Seyyedali in \cite[\S\,6.2]{KMS16}.
On manifolds with ample canonical line bundle,
the uniform convergence of canonically balanced
metrics to the polarized Kähler-Einstein metric,
which always exists
in that case, follows from works of Tsuji \cite{Tsu10}
and Berndtsson.
Our method gives smooth convergence, and also establishes
the uniform convergence for anticanonically balanced metrics
on Fano manifolds. Finally, our method also applies to the
case of coupled K\"ahler-Einstein metrics considered by Takahashi
in \cite{Tak21}.

The adaptation
of our proof for the original notion
of balanced
metrics requires a refined
estimate on the spectral gap of the quantum channel,
which we establish in \cite[Th. 4.11]{IP21}.
Note that the use of the Kähler-Einstein Laplacian,
which is of order two, replaces in the anticanonical setting
the use of the \emph{Lichnerowicz
operator}, which is of order four,
in the original proof of Donaldson.
On the other hand, following the works
of Berman and Witt Nyström in \cite{BW14} and Takahashi in
\cite{Tak15}, we use in \cite{Ioo21} the method of the present
paper to handle the case of general
automorphism groups, replacing Kähler–Einstein metrics by Kähler–Ricci solitons. 
Finally, we also hope to apply our method to the case
of \emph{metaplectically balanced metrics}, giving
an approximation of the \emph{Cahen-Gutt moment map}
and involving a differential operator of order six,
following the program of Futaki and La Fuente-Gravy
outlined in \cite{FL19,FO20}.

The theory of Berezin-Toeplitz
quantization has first been developped by Bordemann,
Meinrenken and
Schlichenmaier in \cite{BMS94}, using the work of
Boutet de Monvel and Sjöstrand on the Szegö kernel in \citep{BdMS75} and the theory of
Toeplitz structures of Boutet de Monvel and Guillemin in
\cite{BdMG81}. This paper is based instead on the theory of
Ma and Marinescu in \cite{MM08b}, which uses the off-diagonal
asymptotic expansion of the Bergman kernel established
by Dai, Liu and Ma in \cite[Th.\,4.18']{DLM06} and which 
holds for an arbitrary volume form in the Hilbert product
\cref{L2intro}.
A comprehensive introduction of this theory can be found in
the book \cite{MM07}. The point of view of quantum measurement
theory on Berezin-Toeplitz quantization, which we adopt
in this paper, has been advocated by Polterovich
in \cite{Pol12,Pol14}.

\section{General Set-up}
\label{setting}


In this Section, we consider a compact complex manifold $X$
with $\dim_\C X=n$ endowed with an ample line bundle $L$,
together with a smooth map
\begin{equation}\label{nuhdef}
\begin{split}
\nu:\Met^+(L)&\longrightarrow\MM(X)\\
h&\longmapsto d\nu_h\,,
\end{split}
\end{equation}
from the space $\Met^+(L)$
of positive Hermitian metrics on $L$ to the space
$\MM(X)$ of smooth volume forms
over $X$. Such a map is called a \emph{volume map}.
For any $h\in\Met^+(L)$,
we write $\Vol(d\nu_h)>0$ for the volume of
$d\nu_h\in\MM(X)$.

For any $h\in\Met^+(L)$ and $p\in\N$, we write
$h^p\in\Met^+(L^p)$ for the induced positive Hermitian metric
on the $p$-th tensor power $L^p$. Conversely, any
$h^p\in\Met^+(L^p)$ uniquely determines
a positive Hermitian metric $h\in\Met^+(L)$.
We write $\cinf(X,L^p)$
for the space of smooth sections of $L^p$
and
\begin{equation}
H^0(X,L^p)\subset\cinf(X,L^p)
\end{equation}
for the subspace of
holomorphic sections of $L^p$ over $X$. We set
\begin{equation}\label{npdef}
n_p:=\dim H^0(X,L^p)\,.
\end{equation}

%
%

\subsection{Balanced metrics}
\label{balsec}

%


Recall from the classical \emph{Kodaira embedding theorem} that
a holomorphic line bundle $L$ is ample if and only if
for all $p\in\N$ big enough, the evaluation map
$\ev_x:H^0(X,L^p)\to L^p_x$
is surjective for all $x\in X$ and the
induced \emph{Kodaira map}
\begin{equation}\label{Kod}
\begin{split}
\text{Kod}_{p}:X&\longrightarrow\mathbb{P}(H^0(X,L^p)^*)\,,\\
x\,&\,\longmapsto \{\,s\in H^0(X,L^p)~|~s(x)=0\,\}
\end{split}
\end{equation}
is an embedding. In this section, we fix such a $p\in\N$.

We denote by $\Prod(H^0(X,L^p))$ the space of
Hermitian inner products on $H^0(X,L^p)$, and for any
$H\in\Prod(H^0(X,L^p))$, we denote by $\cL(H^0(X,L^p),H)$ the
space of endomorphisms on $H^0(X,L^p)$ which are
Hermitian with respect to $H$.
In the following definition, we introduce
the basic tools of this paper.
Their names will be justified in the next Section.

\begin{defi}\label{cohstateprojdef}
The \emph{coherent state
projector} associated to $H\in\Prod(H^0(X,L^p))$ is the map
\begin{equation}
\Pi_H:X\longrightarrow\cL(H^0(X,L^p),H)
\end{equation}
sending $x\in X$ to the orthogonal projector with
respect to $H$ satisfying
\begin{equation}\label{cohstateprojfla}
\Ker\Pi_H(x)=\{\,s\in H^0(X,L^p)~|~s(x)=0\,\}\,.
\end{equation}
The \emph{Berezin symbol} associated to
$H\in\Prod(H^0(X,L^p))$ is the map
\begin{equation}\label{BTdequantfla}
\begin{split}
\sigma_H:\cL(H^0(X,L^p),H)&\longrightarrow\cinf(X,\R)\\
A~&\longmapsto~\Tr[A\Pi_H]\,.
\end{split}
\end{equation}
\end{defi}

Note that the subspace \cref{cohstateprojfla} is the
hyperplane $\Kod_p(x)\subset H^0(X,L^p)$ given by the Kodaira
map \cref{Kod}, and the coherent state projector $\Pi_H(x)$
is thus a rank-$1$ projector, for all $x\in X$.

Recall that
$L^p$ is identified with the pullback of the dual
\emph{tautological
line bundle} over $\IP(H^0(X,L^p))$
via the Kodaira map \cref{Kod}.
Thus given $H\in\Prod(H^0(X,L^p))$, the
induced \emph{Fubini-Study metric} on the dual of the 
tautological line bundle pulls back to a positive Hermitian
metric on $L^p$. Using the coherent state projector
of \cref{cohstateprojdef}, this translates into the following
definition.

\begin{defi}\label{FSdef}
The \emph{Fubini-Study map} is the map
\begin{equation}\label{FSdeffla}
\FS:\Prod(H^0(X,L^p))\longrightarrow\Met^+(L^p)\,,
\end{equation}
sending $H\in\Prod(H^0(X,L^p))$ to the positive
Hermitian metric 
$\FS(H)\in\Met^+(L^p)$ on $L^p$
defined for any $s_1,\,s_2\in H^0(X,L^p)$ and $x\in X$ by
\begin{equation}\label{hFSdef}
\<s_1(x),s_2(x)\>_{\FS(H)}:=\<\Pi_H(x)\,s_1,\,s_2\>_H\,.
\end{equation}
%
\end{defi}

Recall on the other hand the definition \cref{L2intro}
of the \emph{Hilbert map}
\begin{equation}\label{Hilbdef}
\Hilb_\nu:\Met^+(L^p)\longrightarrow\Prod(H^0(X,L^p))\,,
\end{equation}
which holds for a general volume map \cref{nuhdef}.
We are now ready to introduce the main concept of this paper.

\begin{defi}\label{baldef}
A Hermitian metric $h^p\in\Met^+(L^p)$ is called
\emph{balanced} with respect to $\nu:\Met^+(L)\to\MM(X)$
if it satisfies
\begin{equation}
\FS\circ\Hilb_\nu(h^p)=h^p\,.
\end{equation}
A Hermitian product $H\in\Prod(H^0(X,L^p))$ is called
\emph{balanced} with respect to
$\nu:\Met^+(L)\to\MM(X)$ if it satisfies
\begin{equation}
\Hilb_\nu\circ\FS(H)=H\,.
\end{equation}
\end{defi}

Note that if $H\in\Prod(H^0(X,L^p))$ is a balanced product,
then $\FS(H)\in\Met^+(L^p))$ is a balanced metric,
and conversely, if $h^p\in\Met^+(L^p)$ is a balanced metric,
then $\Hilb_\nu(h^p)\in\Prod(H^0(X,L^p))$ is a balanced
product.
%

\begin{ex}\label{Liouvilleex}
The most fundamental example of a volume map is the
\emph{Liouville volume map}
\begin{equation}\label{Liouville}
\begin{split}
\nu:\Met^+(L)&\longrightarrow\MM(X)\\
h&\longmapsto d\nu_h:=\frac{\om_h^n}{n!}\,.
\end{split}
\end{equation}
Note that in that case, the volume $\Vol(X,L):=\Vol(d\nu_h)>0$
does not depend on $h\in\Met^+(L)$.
The analogue of \cref{mainth} in this context, where the limit
metric is a polarized
Kähler metric of \emph{constant scalar curvature},
has been established by
Donaldson in \cite{Don01}. The simple convergence of
the associated Donaldson
iterations as in \cref{donsec}
has been established by Donaldson in \cite{Don05}
and Sano in \cite[Th.\,1.2]{San06}.
\end{ex}

\begin{ex}\label{cstantex}
The simplest example of a volume map is the volume
map with a constant value $d\nu\in\MM(X)$ not depending on
$h\in\Met^+(L)$.
Balanced metrics in this context are called
\emph{$\nu$-balanced metrics}, and have first been studied
by Bourguignon, Li and Yau in \cite{BLY94}.
Donaldson apply them in \cite{Don09}
to study the polarized \emph{Yau metric} \cite{Yau78}
associated to $d\nu$,
which always exists and
is defined as the unique polarized Kähler metric
such that
\begin{equation}
\frac{\om_h^n}{n!}=c\,d\nu\,,
\end{equation}
for some multiplicative constant $c>0$.
This is of specific interest in case
$X$ is a \emph{Calabi-Yau manifold,} so that its canonical
line bundle $K_X$ is trivial and one can
take $d\nu:=\sqrt{-1}^{\,n^2}\,\theta\wedge\overline{\theta}$,
where $\theta\in H^0(X,K_X)$ is the unique nowhere vanishing
section of $K_X$ up to a multiplicative constant.
Then the polarized Yau metric coincides with the
polarized \emph{Ricci flat
metric}. The smooth convergence of the $\nu$-balanced metrics
towards the Yau metric as $p\to+\infty$
has been established by Donaldson in
\cite[\S\,2.2]{Don09}
and by Keller in \cite{Kel09}.
In that case, 
the assumption on the automorphism group is not needed.
The simple convergence of the associated Donaldson
iterations as in \cref{donsec} has been
established by Donaldson in \cite[Prop.\,4]{Don09}, and
exponential convergence as well as
the asymptotics of the optimal
rate of convergence have been worked out in
\cite[Th.\,3.1, Rmk.\,4.12]{IKPS19}.
\end{ex}

\begin{ex}\label{canex}
In case the \emph{canonical line bundle} $L:=K_X$ of $X$
is ample,
one can consider the \emph{canonical volume map}, sending
a positive Hermitian metric $h\in\Met(K_X)$ to the induced volume form defined analogously to \cref{dnucandef}
over any contractible
$U\subset X$ via a non-vanishing $\theta\in\cinf(U,K_X)$ by
\begin{equation}
d\nu_h:=\sqrt{-1}^{\,n^2}\,
\frac{\theta\wedge\overline{\theta}}{|\theta|_h^2}\,.
\end{equation}
In that case, the polarized
Kähler-Einstein metric always exists by a result of Aubin
\cite{Aub78} and Yau \cite{Yau78}.
The uniform convergence of balanced metrics
to the Kähler-Einstein metric as $p\to+\infty$
in this setting has been established by Tsuji \cite{Tsu10}
and Berndtsson (see also \cite[Th.\,7.1]{BBGZ13} for another
proof of the
convergence in the weak sense of currents). 
Once again, the assumption on the automorphism group is not
needed in that case.

The dual version, when $L:=K_X^*$ is ample, uses the
\emph{anticanonical volume map} \cref{dnucandef}.
\cref{mainth}
on the smooth convergence of the balanced metrics
to the polarized Kähler-Einstein metric
as $p\to+\infty$ in this setting is the main result
of this paper.
The exponential convergence of Donaldson's
iterations in this context is the result of \cref{gapth}.
Note that in this case,
and by contrast with the case $K_X$ ample described above,
even if we assume that the automorphism group is discrete,
Tian showed in \cite{Tia97}
that a Kähler-Einstein metric does not exist in
general.
%
\end{ex}

\subsection{Berezin-Toeplitz quantization}
\label{BTsec}

In this Section, we fix a positive Hermitian metric
$h\in\Met^+(L)$ and assume that $p\in\N$ is big enough
so that the Kodaira map \cref{Kod} is well defined and
an embedding. We consider the
Hermitian product $L^2(h^p)\in\Prod(H^0(X,L^p))$ defined for any
$s_1,\,s_2\in\cinf(X,L^p)$ by
\begin{equation}\label{L2}
\<s_1,s_2\>_{L^2(h^p)}:=
\int_X\,\<s_1(x),s_2(x)\>_{h^p}\,d\nu_h(x)\,.
\end{equation}
We write
\begin{equation}
\HH_p:=\left(H^0(X,L^p),\<\cdot,\cdot\>_{L^2(h^p)}\right)\,,
\end{equation}
for the associated Hilbert space of holomorphic sections.
We write $\cL(\HH_p)$ for the space of Hermitian endomorphisms
of $\HH_p$, and
\begin{equation}
\Pi_p:X\longrightarrow\cL(\HH_p)\,,
\end{equation}
for the associated coherent projector of \cref{cohstateprojdef}.
From the point of view of quantum mechanics, this coherent
state projector induces a \emph{quantization} of the
symplectic manifold $(X,\om_h)$, seen as a classical phase space.
A fundamental property in this respect is the following
result.

\begin{prop}\label{Rawnprop}
There exists a unique positive function
$\rho_{h^p}\in\cinf(X,\R)$, called the \emph{Rawnsley
(or density of states) function},
such that for any $s_1,\,s_2\in\HH_p$ and $x\in X$, we have
\begin{equation}\label{Rawndeffla}
\rho_{h^p}(x)\,\<\Pi_p(x)s_1,s_2\>_{L^2(h^p)}
=\<s_1(x),s_2(x)\>_{h^p}\,.
\end{equation}
In particular, we have
\begin{equation}\label{W(X)=1}
\int_X\,\Pi_p(x)\,\rho_{h^p}(x)\,d\nu_h(x)=\Id_{\HH_p}\,.
\end{equation}
\end{prop}
\begin{proof}
For any $x\in X$, consider the associated evaluation map
$\ev_x:\HH_p\fl L^p_x$,
and write $\ev_x^*:L^p_x\fl \HH_p$ for its dual with respect
to $h^p$ and $L^2(h^p)$.
Then for any $s_1,\,s_2\in\HH_p$,
we have by definition
\begin{equation}
\<s_1(x),s_2(x)\>_{h^p}=\<\ev_x^*\ev_x s_1,s_2\>_{L^2(h^p)}\,.
\end{equation}
By \cref{cohstateprojdef},
the endomorphisms $\ev_x^*\ev_x$ and $\Pi_p(x)$ have same kernel
in $\HH_p$, given by the hyperplane
$\Kod_p(x)\subset H^0(X,L^p)$ image of $x\in X$ by the Kodaira
map \cref{Kod}. As they are both
Hermitian, they also have same $1$-dimensional image in $\HH_p$,
so that there exists a unique
positive number
$\rho_{h^p}(x)>0$ such that
\begin{equation}
\rho_{h^p}(x)\,\Pi_p(x)=\ev_x^*\ev_x\,.
\end{equation}
As they both depend smoothly on $x\in X$,
this defines a unique smooth positive
function $\rho_{h^p}\in\cinf(X,\R)$ satisfying formula
\cref{Rawndeffla}. The identity \cref{W(X)=1} then follows
by integrating formula \cref{Rawndeffla} against
$d\nu_h$ via the definition \cref{L2} of $L^2(h^p)$.
\end{proof}
%

The fundamental role played by
the Rawnsley function in the study of
the balanced metrics of \cref{baldef} comes from the
following basic result.

\begin{prop}\label{balmet=balemb}
A positive Hermitian metric $h^p\in\Met^+(L^p)$ is balanced
with respect to $\nu:\Met^+(L)\to\MM(X)$ if and only if
for all $x\in X$,
the associated Rawnsley function $\rho_{h^p}\in\cinf(X,\R)$
satisfies
\begin{equation}\label{Rpcstant}
\rho_{h^p}(x)=\frac{n_p}{\Vol(d\nu_h)}\,.
\end{equation}
\end{prop}
\begin{proof}
By definition, we have
\begin{equation}\label{Hilb=L2}
\Hilb_\nu(h^p)=\frac{n_p}{\Vol(d\nu_h)}\,L^2(h^p)\,,
\end{equation}
so that by \cref{FSdef} and 
\cref{Rawnprop}, for all
$s_1,\,s_2\in H^0(X,L^p)$ and $x\in X$ we have
\begin{equation}\label{Rdveq}
\rho_{h^p}(x)\,\<s_1(x),s_2(x)\>_{\FS(\Hilb_\nu(h^p))}=
\frac{n_p}{\Vol(d\nu_h)}\,\<s_1(x),s_2(x)\>_{h^p}\,.
\end{equation}
This gives the result by \cref{baldef} of a balanced metric.
\end{proof}

\cref{Rawnprop} describes fundamental
properties of a \emph{coherent state quantization},
given in our context by the following Definition.

\begin{defi}\label{BTquantdef}
The \emph{Berezin-Toeplitz quantization map}
is defined by
\begin{equation}\label{BTmapfla}
\begin{split}
T_{h^p}:\cinf(X,\R)&\longrightarrow\cL(\HH_p)\,.\\
f\,&\longmapsto\,\int_X f(x)\,\Pi_p(x)\,\rho_{h^p}(x)\,d\nu_h(x)
\end{split}
\end{equation}
\end{defi}

Using \cref{Rawnprop}, we have the following characterization
of the Berezin-Toeplitz quantization of $f\in\cinf(X,\R)$,
for all $s_1,\,s_2\in\HH_p$,
\begin{equation}\label{Tpf}
\begin{split}
\<T_{h^p}(f)s_1,s_2\>_{L^2(h^p)}&=\int_X\,f(x)\,\<\Pi_p(x)s_1,s_2\>_{L^2(h^p)}\,\rho_{h^p}(x)\,d\nu_h(x)\\
&=\int_X\,f(x)\,\<s_1(x),s_2(x)\>_{h^p}\,d\nu_h(x)\,.
\end{split}
\end{equation}
This shows that \cref{BTquantdef} coincides with the usual
definition of Berezin-Toeplitz quantization
associated with the volume form
$d\nu_h\in\MM(X)$, as described in \cite[Chap.\,7]{MM07}.
In the same way, one readily checks
that the Rawnsley function of \cref{Rawnprop}
coincides with the associated
\emph{Bergman kernel} along the diagonal,
as described in \cite[Chap.\,4]{MM07}. We will give
in \cref{sumprop}
its geometric description as a density of states.

In the context of quantization,
the Berezin symbol \cref{cohstateprojfla}
of a quantum observable $A\in\cL(\HH_p)$ is interpreted
as the classical observable given by
the expectation value of $A$
at coherent states.
This gives rise to the following concept, which will be the
main tool of this paper.

\begin{defi}\label{quantchandef}
The \emph{Berezin-Toeplitz quantum channel} is the
linear operator
\begin{equation}
\begin{split}
\EE_{h^p}:\cL(\HH_p)&\longrightarrow\cL(\HH_p)\,,\\
A~&\longmapsto~T_{h^p}\left(\sigma_{L^2(h^p)}
\left(A\right)\right)\,.
\end{split}
\end{equation}
\end{defi}

In the context of quantum measurement theory,
the quantum channel describes the effect of a measurement
on quantum observables.
The basic properties of the Berezin-Toeplitz quantum channel
have been studied in \cite{IKPS19}, based on \cite{BMS94}.
They are summarized in the following proposition.

\begin{prop}\label{quantchanprop}
The Berezin-Toeplitz quantum channel $\EE_{h^p}$ is a
positive self-adjoint
operator on the real Hilbert space $\cL(\HH_p)$
equipped with the trace norm,
and its eigenvalues $\{\gamma_k(h^p)\}_{k=1}^{n_p^2}$
counted with multiplicities satisfy
\begin{equation}\label{specE}
1=\gamma_0(h^p)>\gamma_1(h^p)\geq\gamma_2(h^p)
\geq\cdots\geq\gamma_{n_p^2}(h^p)>0\,,
\end{equation}
where $1=\gamma_0(h^p)$ is associated with
the eigenvector $\Id_{\HH_p}\in\cL(\HH_p)$.
\end{prop}
\begin{proof}
By \cref{cohstateprojdef,BTquantdef},
for any $A,\,B\in\cL(\HH_p)$,
we have
\begin{equation}\label{quatnchantr}
\Tr\left[A\,\EE_{h^p}\left(B\right)\right]=\int_X\,\Tr[A\Pi_p(x)]
\,\Tr[B\Pi_p(x)]\,\rho_{h^p}(x)\,d\nu_h(x)\,,
\end{equation}
so that as $\rho_{h^p}>0$ by definition,
the quantum channel $\EE_{h^p}$ is positive
and self-adjoint
for the trace norm on $\cL(\HH_p)$.
Furthermore, as $\Tr[\Pi_p(x)]=1$ for all $x\in X$, we see
from \cref{Rawnprop}
that $\EE_{h^p}(\Id_{\HH_p})=\Id_{\HH_p}$. 
The injectivity of $\EE_{h^p}$ and the fact that
$\gamma_1(h^p)<1$ follow from the results of
\cite[Ex.\,4.1, Props.\,4.7,\,4.8]{IKPS19}.
\end{proof}
%


The positive number $\gamma:=1-\gamma_1(h^p)>0$ is called
the \emph{spectral gap} of the quantum channel, and it
measures the loss of information associated with
repeated quantum measurements. The following estimate
on its \emph{semi-classical limit} as $p\fl+\infty$ is central
to this paper.

\begin{theorem}\label{Bpgap}
{\cite[Th.\,3.1, Rmk.\,3.12]{IKPS19}}
There exists a constant $C>0$ such
that for all $p\in\N$, we have
\begin{equation}
\left|1-\gamma_1(h^p)-\frac{\lambda_1(h)}{4\pi p}\right|\leq
\frac{C}{p^2}\,,
\end{equation}
where $\lambda_1(h)>0$ is the first positive eigenvalue of the
Riemannian Laplacian of $(X,g_h^{TX})$ acting on $\cinf(X,\C)$.

Moreover, there exists $l\in\N$ such that
for any bounded subset $K\subset\Met^+(L)$
in $\CC^l$-norm over which the volume map
\cref{nuhdef} is bounded from below,
the constant $C>0$ can be chosen uniformly in
$h\in K$.
\end{theorem}

The uniformity in the metric is not explicitly stated in
\cite[Th.\,3.1]{IKPS19}, but as noted in 
\cite[Rmk.\,4.9]{IKPS19}, it readily follows from the
uniformity in the metric of the estimates on the Bergman
kernel of \cite[Th.\,4.18']{DLM06}.

Furthermore, as explained in \cite[Rmk.\,3.12]{IKPS19},
the case of a general volume form $d\nu_h\in\MM(X)$
follows from a trick due to
Ma and Marinescu in \cite[\S\,4.1.9]{MM07}.
This trick is based on the fact that the
$L^2$-product \cref{L2} coincides with the
$L^2$-product associated with the Liouville form
$\om_h^n/n!$ and the Hermitian metric
$h^p\otimes h^E$ on $L^p\otimes E$,
where $E=\C$ is the trivial line bundle and
$h^E\in\Met^+(E)$ is defined by
$|1|_{h^E}^2\,\om_h^n/n!:=d\nu_h$.
This implies in particular that the
Rawnsley function $\til{\rho}_{h^p}\in\cinf(X,\R)$ 
associated with $\om_h^n/n!$ and $h^p\otimes h^E$ as above
satisfies
\begin{equation}
\rho_{h^p}\,d\nu_h=\til{\rho}_{h^p}\,\frac{\om_h^n}{n!}\,.
\end{equation}
This gives
the following version of a classical result on the
asymptotics as $p\to+\infty$
of the Rawnsley function, which is the other
crucial estimate needed in this paper.

\begin{theorem}\label{Bergdiagexp}
{\cite[Th.\,1.3]{DLM06}}
There exist functions
$b_r(h)\in\cinf(X,\R)$ for all $r\in\N$
such that for any $m,\,k\in\N$, there exists $C_{m,k}>0$
such that for all $p\in\N$ big enough,
\begin{equation}\label{Bergdiagexpfla}
\left|\,\rho_{h^p}-p^n\sum_{r=0}^{k-1}\frac{1}{p^r}b_r(h)\,
\right|_{\CC^m}\leq \frac{C_{m,k}}{p^{k}}\,,
\end{equation}
Furthermore, the functions $b_r(h)\in\cinf(X,\R)$, $r\in\N$,
depend
polynomially on $h\in\Met^+(L)$ and its successive derivatives
along $X$, and the function $b_0(h)\in\cinf(X,\R)$
satisfies the identity
\begin{equation}\label{Bergcoeff}
b_0(h)\,d\nu_h=\frac{\om^n_h}{n!}\,.
\end{equation}
Finally, for each $m,\,k\in\N$, there exists $l\in\N$
such that for any bounded subset $K\subset\Met^+(L)$
in $\CC^l$-norm over which the volume map
\cref{nuhdef} is bounded from below,
the constant $C_{m,k}>0$ can be chosen uniformly in
$h\in K$.
\end{theorem}

In particular, using \cref{Rawnprop} and the fact that
$\Tr[\Pi_p]=1$, \cref{Bergdiagexp} implies
that the dimension of $\HH_p$
satisfies the following estimate as $p\to+\infty$,
\begin{equation}\label{npexp}
n_p=\Tr[\Id_{\HH_p}]=
\int_X\,\rho_{h^p}(x)\,d\nu_h(x)=p^n\Vol(X,L)+O(p^{n-1})\,,
\end{equation}
where $\Vol(X,L)>0$ is the volume of the Liouville volume map
\cref{Liouville}, which does not depend
on $h\in\Met^+(L)$.

%
%
\subsection{Moment map}
\label{momentsec}

In this Section, we fix $p\in\N$ big enough so that the
Kodaira map \cref{Kod} is well defined and an embedding,
and we consider the space
$\BB(H^0(X,L^p))$ of bases of $H^0(X,L^p)$. For any $\s\in\BB(H^0(X,L^p))$,
we write $H_\s\in\Prod(H^0(X,L^p))$
for the Hermitian product for which it is an orthonormal
basis,
and write $h_\s\in\Met^+(L)$ for the positive Hermitian
metric
defined through \cref{FSdef} by the formula
\begin{equation}\label{hsdef}
h_\s^p:=\FS(H_\s)\in\Met^+(L^p)\,.
\end{equation}
Write $\Herm(\C^{n_p})$ for the space of Hermitian
matrices of $\C^{n_p}$.
The following central tool in the study of
balanced metrics has been
introduced by Donaldson \cite{Don01,Don09} in his moment
map picture for the study of canonical Kähler metrics.

\begin{defi}\label{momentdef}
The \emph{moment map} associated to
$\nu:\Met^+(L)\to\MM(X)$ is the map
\begin{equation}
\mu_\nu:\BB(H^0(X,L^p))\longrightarrow\Herm(\C^{n_p})
\end{equation}
defined for all $\textbf{s}=\{s_j\}_{j=1}^{n_p}
\in\BB(H^0(X,L^p))$ by the formula
\begin{equation}
\mu_\nu(\textbf{s})
:=\left(\int_X\,
\<s_j(x),s_k(x)\>_{h^p_\s}\,
d\nu_{h_\textbf{s}}(x)\right)_{j,\,k=1}^{n_p}
-\frac{\Vol(d\nu_{h_\textbf{s}})}{n_p}\Id_{\C^{n_p}}\,.
\end{equation}
\end{defi}

The fundamental role of this moment map in the study of
the balanced products of \cref{baldef} comes from the
following basic result.

\begin{prop}\label{momentbal}
For any $\s\in\BB(H^0(X,L^p))$, the induced Hermitian product
$H_\s\in\Prod(H^0(X,L^p))$ is balanced with respect to
$\nu:\Met^+(L)\to\MM(X)$ if and only if
\begin{equation}
\mu_\nu(\s)=0\,.
\end{equation}
\end{prop}
\begin{proof}
Comparing \cref{FSdef} and formula
\cref{L2intro} with \cref{momentdef} and
formula \cref{hsdef}, we see that
$\s=\{s_j\}_{j=1}^{n_p}\in\BB(H^0(X,L^p))$
satisfies $\mu_\nu(\s)=0$ if and only if
\begin{equation}
\left(\<s_j,s_k\>_{\Hilb_\nu\left(\FS(H_\s)\right)}\right)_{j,\,k=1}
^{n_p}=\Id_{\C^{n_p}}\,,
\end{equation}
i.e. if and only if $\s=\{s_j\}_{j=1}^{n_p}$ is an orthonormal
basis for $\Hilb_\nu\left(\FS(H_\s)\right)\in\Prod(H^0(X,L^p))$.
But this property characterizes $H_\s\in\Prod(H^0(X,L^p))$,
so that $\mu_\nu(\s)=0$ if and only if
\begin{equation}
\Hilb_\nu\left(\FS(H_\s)\right)=H_\s\,,
\end{equation}
which is the \cref{baldef} of a balanced product.
\end{proof}

In the following proposition, we give useful characterizations
for the Fubini-Study metric of \cref{FSdef} and
the Rawnsley function of \cref{Rawnprop}
in terms of bases of $H^0(X,L^p)$, recovering
their familiar descriptions in this context.

\begin{prop}\label{sumprop}
For any $h^p\in\Met^+(L^p)$,
the associated Rawnsley function
$\rho_{h^p}\in\cinf(X,L^p)$
is given for any $x\in X$ by the formula
\begin{equation}\label{Rawnsumfla}
\rho_{h^p}(x)=\sum_{j=1}^{n_p}\,|s_j(x)|^2_{h^p}\,,
\end{equation}
where $\{s_j\}_{j=1}^{n_p}\in\BB(H^0(X,L^p))$
is an orthonormal basis for $L^2(h^p)$.

For any basis $\s=\{s_j\}_{j=1}^{n_p}\in\BB(H^0(X,L^p))$,
the induced
Fubini-Study metric $h_\s^p\in\Met^+(L^p)$
is characterized by the following formula, for any $x\in X$,
\begin{equation}\label{FSsumfla}
\sum_{j=1}^{n_p}|s_j(x)|^2_{h^p_\s}=1\,.
\end{equation}
In particular, we have $\Tr[\mu_\nu(\s)]=0$
for all $\s\in\BB(H^0(X,L^p))$.
\end{prop}
\begin{proof}
By \cref{Rawnprop}, if $\{s_j\}_{j=1}^{n_p}\in\BB(H^0(X,L^p))$
is an orthonormal basis for $L^2(h^p)$,
then we have
\begin{equation}
\sum_{j=1}^{n_p}\,|s_j|^2_{h^p}=
\sum_{j=1}^{n_p}\rho_{h^p}\,\<\Pi_ps_j,s_j\>_{L^2(h^p)}
=\rho_{h^p}\,\Tr[\Pi_p]=\rho_{h^p}\,,
\end{equation}
which shows formula \cref{Rawnsumfla}.
On the other hand, any
$\s=\{s_j\}_{j=1}^{n_p}\in\BB(H^0(X,L^p))$ is by definition
an orthonormal
basis for $H_\s\in\Prod(H^0(X,L^p))$, so that
by \cref{FSdef} we get
\begin{equation}
\sum_{j=1}^{n_p}|s_j|^2_{h^p_\s}=
\sum_{j=1}^{n_p}\<\Pi_{H_\s}s_j,s_j\>_{H_\s}
=\Tr[\Pi_{H_\s}]=1\,,
\end{equation}
which clearly characterizes $h^p_\s\in\Met^+(L^p)$.
By \cref{momentdef}, this readily implies $\Tr[\mu_\nu(\s)]=0$.
\end{proof}

Recall \cref{cohstateprojdef} for the Berezin symbol associated
with a Hermitian product $H\in\Prod(H^0(X,L^p))$.

\begin{prop}\label{FSvar}
For any $\textbf{s}\in\BB(H^0(X,L^p))$ and
$B\in\cL(H^0(X,L^p),H_\s)$,
we have
\begin{equation}\label{hpeBsfla}
\sigma_{H_\s}(e^{2B})\,h^p_{e^B\s}=h^p_{\s}\,.
\end{equation}
\end{prop}
\begin{proof}
By \cref{cohstateprojdef,FSdef},
for any $B\in\cL(H^0(X,L^p),H_\s)$ and
writing $\textbf{s}=\{s_j\}_{j=1}^{n_p}$, we have
\begin{equation}\label{G*GBersymb}
\begin{split}
\sigma_{H_\s}(e^{2B})
&=\Tr[e^B\Pi_{H_\s}e^B]\\
&=\sum_{j=1}^{n_p}\<\Pi_{H_\s}e^Bs_j,e^Bs_j\>_{H_\s}
=\sum_{j=1}^{n_p}\left|e^{B}s_j\,\right|_{h^p_\s}^2\,.
\end{split}
\end{equation}
As $\{e^Bs_j\}_{j=1}^{n_p}$ is an orthonormal
basis for $H_{e^B\s}\in\Prod(H^0(X,L^p))$ by definition,
this shows the result by the characterization
of the Fubini-Study metric given in \cref{sumprop}.
\end{proof}

Consider now the free and transitive action of
$\GL(\C^{n_p})$ on $\BB(H^0(X,L^p))$ via the formula
\begin{equation}\label{GLact}
G.\s:=\left\{\sum_{k=1}^{n_p} G_{jk}s_k\right\}_{j=1}^{n_p}\,,
\end{equation}
for any $G=(G_{jk})_{j,\,k=1}^{n_p}\in\GL(\C^{n_p})$
and $\s=\{s_j\}_{j=1}^{n_p}\in\BB(H^0(X,L^p))$.
By derivation, this induces a canonical
identification of tangent spaces
\begin{equation}\label{isomEnd}
T_\s\BB(H^0(X,L^p))\simeq\End(\C^{n_p})\,,
\end{equation}
making $\BB(H^0(X,L^p))$ into a complete Riemannian manifold
via the Hermitian product
defined on $A,\,B\in\End(\C^{n_p})$ by the formula
\begin{equation}\label{trnorm}
\<A,B\>_{tr}:=\Tr[AB^*]\,.
\end{equation}
Restricting to Hermitian matrices
$\Herm(\C^{n_p})\subset\End(\C^{n_p})$,
this induces for all $\s\in\BB(H^0(X,L^p))$
an isometry
\begin{equation}\label{isomHerm}
\Herm(\C^{n_p})\simeq\cL(H^0(X,L^p),H_\s)\,.
\end{equation}
The unitary group
$U(n_p)\subset\GL(\C^{n_p})$ acts by isometries
on $\BB(H^0(X,L^p))$, and
the quotient map
\begin{equation}\label{symmap}
\begin{split}
\BB(H^0(X,L^p))&\longrightarrow\Prod(H^0(X,L^p))\\
\s~&\longmapsto~H_\s
\end{split}
\end{equation}
makes in turn
$\Prod(H^0(X,L^p))$ into
a complete Riemannian manifold, whose geodesics are of the
form
\begin{equation}\label{geod}
t\longmapsto H_{e^{tA}\s}\in\Prod(H^0(X,L^p)),\,t\in\R\,,
\end{equation}
for all $A\in\Herm(\C^{n_p})$. 

We will write $\Pi_\s:X\to\Herm(\C^{n_p})$ and $\sigma_\s:\Herm(\C^{n_p})\to\cinf(X,\R)$ for the coherent state
projector and the Berezin symbol of
\cref{cohstateprojdef} associated with $H_\s\in\Prod(H^0(X,L^p))$
under the identification \cref{isomHerm}
induced by any $\s\in\BB(H^0(X,L^p))$.
In these notations, we have the following comparison formula for
the Berezin symbols associated with two different bases
in the corresponding identifications.
\begin{prop}\label{FSvar2}
For any $A,\,B\in\Herm(\C^{n_p})$ and $\s\in\BB(H^0(X,L^p))$,
we have
\begin{equation}
\sigma_{e^B\s}(A)=\sigma_{\s}(e^{2B})^{-1}
\,\sigma_{\s}(e^BAe^B)\,.
\end{equation}
\end{prop}
\begin{proof}
Write $\s=:\{s_j\}_{j=1}^{n_p}$
and $e^B\s=:\{\til{s}_j\}_{j=1}^{n_p}$, so that
by definition \cref{GLact}
of the action and
writing $e^B:=(G_{jk})_{j,k=1}^{n_p}$, we have
$\til{s}_j=\sum_{k=1}^{n_p}G_{jk}s_k$ for all $1\leq j\leq n_p$.
Then using \cref{FSdef}, \cref{FSvar} and the fact that
$e^B\in\GL(\C^{n_p})$ is Hermitian, we get
\begin{equation}
\begin{split}
\sigma_{e^B\s}(A)&=\sum_{j,k=1}^{n_p}
\<A_{jk}\til{s}_k,\til{s}_j\>_{h^p_{e^B\s}}=\sigma_{\s}(e^{2B})^{-1}
\sum_{j,k,l,m=1}^{n_p}\<A_{jk}G_{kl}s_l,
G_{jm}s_m\>_{h^p_{\s}}\\
&=\sigma_{\s}(e^{2B})^{-1}
\sum_{l,m=1}^{n_p}\left(e^BAe^B\right)_{ml}\<s_l,s_m\>_{h^p_{\s}}\,.
\end{split}
\end{equation}
This implies the result by \cref{cohstateprojdef,FSdef}.
\todo{Typo corrigée: \cref{cohstatecoord}$\mapsto$\cref{FSdef}. Changer?}
\end{proof}

\section{Anticanonically balanced metrics}
\label{anticansec}

In this Section, we consider the general set-up of 
\cref{setting} in the particular case
when $X$ is a Fano manifold,
meaning that its anticanonical line bundle
$K_X^*:=\det(T^{(1,0)}X)$ is ample. We take $L:=K_X^*$
and consider the
anticanonical volume map
$\nu:\Met(K_X^*)\longrightarrow\MM(X)$
defined by formula \cref{dnucandef}.


\subsection{Kähler-Einstein metrics and anticanonical volume map}
\label{KEsec}
%
%
%

A Kähler form $\om\in\Om^2(X,\R)$ on a
compact complex manifold $X$ induces a natural
Hermitian metric $h_\om\in\Met(K_X^*)$,
defined using the anticanonical volume form
\cref{dnucandef} by the formula
\begin{equation}
\frac{\om^n}{n!}=d\nu_{h_\om}\,.
\end{equation}
Conversely, a positive Hermitian metric $h\in\Met^+(K_X^*)$
induces a Kähler form $\om_h\in\Om^2(X,\R)$
as in \cref{preq}, but $\om_{h_\om}$ do not coincide with
$\om$ in general. This motivates
the following important notion of Kähler geometry.

\begin{defi}\label{KEdef}
A positive Hermitian metric $h\in\Met^+(K_X^*)$ is called
\emph{Kähler-Einstein} if there exists a constant $c>0$ such that
the associated Kähler form $\om_h$ satisfies
\begin{equation}
\frac{\om^n_h}{n!}=c\,d\nu_{h}\,.
\end{equation}
The associated polarized Kähler metric $g^{TX}_h$ is then called a
\emph{Kähler-Einstein metric}.
\end{defi}

Let us recall some basic facts about such Kähler-Einstein
metrics, which can be found for instance in
\cite[Chap.\,3\,--\,4]{Sze14}. First of all, for a positive
Hermitian metric $h\in\Met^+(K_X^*)$ and in our convention
\cref{preq} for the associated Kähler form
$\om_h\in\Om^2(X,\R)$, the Kähler-Einstein
condition of \cref{KEdef} is equivalent to the identity
\begin{equation}
\om_h=\frac{1}{2\pi}\Ric(g^{TX}_h)\,,
\end{equation}
where $\Ric(g^{TX}_h)\in\Om^2(X,\R)$ is the \emph{Ricci form}
of $(X,J,g^{TX}_h)$. This implies that the
\emph{scalar curvature} $\scal(g^{TX}_h)$ of
$(X,g^{TX}_h)$ is constant, given by
\begin{equation}
\scal(g^{TX}_h)=4\pi n\,.
\end{equation}
We then have the following classical result of Lichnerowicz and
Matsushima, in a form which can be found in
\cite[Chap.\,3]{Gau} and which will be
a key input in our proof of \cref{mainth}.
Write $\Aut(X)$ for the automorphism group of $X$
as a complex manifold.
\begin{theorem}{\cite{Lic59,Mat57}}
\label{lambda1>4pi}
Assume that $\Aut(X)$ is discrete, and let $h_\infty\in\Met^+(K_X^*)$
be Kähler-Einstein.
Then the first positive eigenvalue $\lambda_1(h_\infty)>0$ of the
Riemannian Laplacian $\Delta_{h_\infty}$
of $(X,g^{TX}_{h_\infty})$ acting on $\cinf(X,\C)$ satisfies
\begin{equation}
\lambda_1(h_\infty)>4\pi\,.
\end{equation}
\end{theorem}

We will often need the
following variation formula for the anticanonical volume
form \cref{dnucandef}.

\begin{prop}\label{dnuef/dnu}
The anticanonical volume form \cref{dnucandef} satisfies
the following formula, for any $f\in\cinf(X,\R)$ and
$h\in\Met(K_X^*)$,
\begin{equation}
d\nu_{e^fh}=e^{f}\,d\nu_h\,.
\end{equation}
\end{prop}
\begin{proof}
If $h^{-1}\in\Met(K_X)$ denotes the Hermitian metric
induced by $h\in\Met(K_X^*)$, then for any $f\in\cinf(X,\R)$ and $h\in\Met(K_X^*)$, we have $(e^fh)^{-1}=e^{-f}h^{-1}$.
This readily implies the result by formula \cref{dnucandef}.
\end{proof}

\begin{rmk}
Let $L:=K_X^*$ be ample, and recall from \cref{momentsec}
that we write $h_\s\in\Met^+(L)$ for the positive Hermitian
metric induced
by the Fubini-Study metric of $H_\s$,
for any $\s=\{s_j\}_{j=1}^{n_p}\in\BB(H^0(X,L^p))$.
Restricted to such metrics, the anticanonical volume form
\cref{dnucandef} admits a metric-independent characterization.
In fact, using \cref{sumprop}, one computes
\begin{equation}
d\nu_{h_{\textbf{s}}}=\left(\sum_{j=1}^{n_p}\,
s_j\otimes\bar{s}_j\right)^{-1/p}\,,
\end{equation}
where the expression inside the parentheses in the last line is
to be considered as a positive section of
$L^p\otimes\overline{L}^p$ equipped
with its natural $\R_+$-structure,
so that its inverse $p$-th root defines a smooth form.
These volume forms have been introduced by Donaldson in
\cite[\S\,2.2.2]{Don09} to approximate numerically
Kähler-Einstein metrics on Fano manifolds, a program
for which
\cref{mainth,gapth} provide a rigorous basis.
\end{rmk}

\subsection{Approximately balanced metrics}
\label{approxbalsec}

Let $X$ be a Fano manifold
with $\Aut(X)$ discrete
and admitting a Kähler-Einstein metric
$h_\infty\in\Met^+(K_X^*)$.
In this Section, we consider the setting of \cref{setting}
with $L:=K_X^*$
for the anticanonical volume map \cref{dnucandef}.

The proof of the following result
is parallel to the proof of the analogous result of Donaldson
\cite[Th.\,26]{Don01} in the case of \cref{Liouvilleex},
replacing the positivity of the Lichnerowicz operator
by \cref{lambda1>4pi}.
All the local $\CC^m$-norms are taken with respect to the fixed Kähler-Einstein metric.

\begin{prop}\label{approxbal}
There exists a sequence of functions
$f_r\in\cinf(X,\R)$, $r\in\N$, such that
for every $k,\,m\in\N$, there exists a constant $C_{k,m}>0$
such that all $p\in\N$ big enough,
the positive Hermitian metric
\begin{equation}\label{approxbalmet}
h_k(p):=\exp\left(\sum_{r=1}^{k-1}\frac{1}{p^r}f_r\right)
h_\infty\in\Met^+(L^p)\,,
\end{equation}
have associated Rawnsley function $\rho_{h_p^k(p)}\in\cinf(X,\R)$
satisfying
\begin{equation}\label{approxbalfla}
\left|\,\rho_{h_k^p(p)}-\frac{n_p}{\Vol(d\nu_{h_k(p)})}\,
\right|_{\CC^m}\leq C_{k,m} p^{n-k}\,.
\end{equation}
\end{prop}
\begin{proof}
First note by \cref{KEdef} of
the Kähler-Einstein metric
$h_\infty\in\Met^+(L)$
that the coefficient $b_0(h_\infty)\in\cinf(X,\R)$ of
\cref{Bergdiagexp} is constant.
This implies the result
for $k=1$.

Let us write $\Delta_{h_\infty}$ for the
Riemannian Laplacian of $(X,g^{TX}_{h^\infty})$
acting on $\cinf(X,\C)$.
Using \cref{dnuef/dnu} and
a classical formula in Kähler geometry,
for any $f\in\cinf(X,\R)$ we get
\begin{equation}\label{varinftyfla}
\dt\Big|_{t=0}\frac{\om_{e^{tf}h_\infty}^n}{\,d\nu_{e^{tf}h_\infty}}
=\left(\frac{1}{4\pi}\Delta_{h_\infty} f-f\right)\frac{\om^n_{h_\infty}}
{d\nu_{h_\infty}}\,.
\end{equation}
Recall by \cref{KEdef} that the Riemannian volume
form of $(X,g_{h_\infty}^{TX})$ is a constant multiple
of $d\nu_{h_\infty}$. Then \cref{lambda1>4pi} shows that
the restriction of the operator
$\left(\frac{1}{4\pi}\Delta_{h_\infty}-1\right)$
admits an inverse on
the orthogonal of the constant functions
inside $L^2(X,\C)$, so that
for any function $f\in\cinf(X,\R)$, there exists a
function $\til{f}\in\cinf(X,\R)$ satisfying
\begin{equation}\label{f1fla}
f-\int_X\,f~
\frac{d\nu_{h_\infty}}{\Vol(d\nu_{h_\infty})}
=\til{f}-\frac{1}{4\pi}\Delta_{h_\infty}\til{f}
\,.
\end{equation}
\todo{pas besoin de la constante:
on a $\Ker\left(\frac{1}{4\pi}\Delta_{h_\infty}-1\right)=\0$}
Take $f:=b_1(h_\infty)\in\cinf(X,\R)$ in \cref{f1fla},
and consider the Rawnsley
function $\rho_{h^p_1(p)}\in\cinf(X,\R)$ associated
with the metric $h_1(p):=e^{\til{f}/p}h_\infty\in\Met^+(L)$.
As $h_1(p)\to h_\infty$ smoothly as $p\to+\infty$
by construction, we can use the uniformity in \cref{Bergdiagexp}
to replace
$h_\infty$ by $h_1(p)$ in
the expansion \cref{Bergdiagexpfla}. 
As the
coefficients in the expansion are polynomials in the
derivatives of $h_1(p)\in\Met^+(L)$, we can take
the Taylor expansion as $p\to+\infty$ of formula 
\cref{Bergcoeff} to get from
formulas \cref{varinftyfla,f1fla} the following
expansion as $p\to+\infty$ in $\CC^m$-norm for all $m\in\N$,
\begin{equation}
\begin{split}
b_0(h_1(p))&+p^{-1}b_1(h_1(p))\\
&=b_0(h_\infty)+p^{-1}\left(\frac{1}{4\pi}\Delta_{h_\infty} \til{f}-\til{f}\right)+p^{-1}b_1(h_\infty)+O(p^{-2})\\
&=b_0(h_\infty)+p^{-1}\int_X\,b_1(h_\infty)~
\frac{d\nu_{h_\infty}}{\Vol(d\nu_{h_\infty})}+O(p^{-2})\,.
\end{split}
\end{equation}
As $b_0(h_\infty)$ is constant by assumption,
this implies that there exists a constant
$C_p>0$ for all
$p\in\N$ such that as $p\to+\infty$ in $\CC^m$-norm
for any $m\in\N$, we have
\begin{equation}\label{approxbalfla1}
\rho_{h_1^p(p)}=C_p+O(p^{n-2})\,,
\end{equation}
and the constant $C_p>0$ is determined up to order $O(p^{-2})$
by taking the integral of both sides against $d\nu_{h_1(p)}$
and using formula \cref{npexp}.
This implies the result for $k=2$.

Let us assume now that the result holds for some $k\in\N$,
so that we have Hermitian metrics
$h_k(p)\in\Met^+(L)$ as in \cref{approxbalmet}
with associated Rawnsley function satisfying the asymptotic
expansion \cref{approxbalfla} as $p\to+\infty$.
As $h_k(p)\to h_\infty$ smoothly as $p\to+\infty$
by construction, we can again
apply \cref{Bergdiagexp} to $\rho_{h_k^p(p)}$,
and taking the Taylor expansion as $p\to+\infty$
of the coefficients $b_r(h_k(p))$ for all $1\leq r\leq k+1$,
we get a sequence of functions $b_r'\in\cinf(X,\R)$
for $1\leq r\leq k$, not depending on $p\in\N$, such that
the asymptotic expansion \cref{Bergdiagexpfla} holds
for these functions. Futhermore, for every $r\leq k-1$,
the function $b_r'$ is constant over $X$ by assumption.
We then take 
\begin{equation}
h_{k+1}(p):=e^{f_k/p^k}h_k(p)\in\Met^+(L^p)
\end{equation}
for all $p\in\N$, where the function $f_k\in\cinf(X,\R)$
is constructed as the function $\til{f}$ of formula \cref{f1fla}
for $f:=b_k'$. One can then repeat the process above
to get the result for $k+1$, which gives the result
for general $k\in\N$ by induction.
\end{proof}

Let us now consider orthonormal bases
$\s_k(p)\in\BB(H^0(X,L^p))$ for the $L^2$-products
$L^2(h_k^p(p))\in\Prod(H^0(X,L^p))$
induced by the Hermitian metrics of \cref{approxbal},
for all $k\in\N$ and
all $p\in\N$ big enough. 
Then under the identification \cref{isomHerm}
and by \cref{sumprop,FSvar}, for any $B\in\Herm(\C^{n_p})$
we have
\begin{equation}\label{heBsp=hkp}
h_{e^B\s_k(p)}^p=\sigma_{\s_k(p)}\left(e^{2B}\right)^{-1}
\rho_{h_k^p(p)}^{-1}\,h_k^p(p)\,.
\end{equation}
The following Lemma is essentially due to Donaldson
\cite[Prop.\,27\,(1)]{Don01},
and we prove it here under our conventions
for convenience.
\begin{lem}\label{approxomkplem}
For any $k,\,k_0,\,m\in\N$ with $k_0>n+1+m/2$,
there exists $C>0$
such that for all $p\in\N$ big enough, we have
\begin{equation}
\label{approxhBp}
\left|\om_{e^{B}\s_k(p)}-\om_\infty\right|_{\CC^{m}}
\leq\frac{C}{p}\,,
\end{equation}
where $\om_{e^B\s_k(p)}$ is the Kähler form induced by
$h_{e^B\s_k(p)}\in\Met^+(L)$ for all $p\in\N$ and $\om_\infty$ is the
Kähler form induced by the Kähler-Einstein metric
$h_\infty\in\Met^+(L)$.
\end{lem}
\begin{proof}

Fix $k\in\N$, and note from \cref{dnuef/dnu} that
$\Vol(d\nu_{h_k(p)})\to\Vol(d\nu_{h_\infty})$
as $p\to+\infty$. Using \cref{approxbal} and the estimate
\cref{npexp} for the dimension,
we know that there is a constant $C>0$
such that for all $p\in\n$, we have
\begin{multline}
\left|\,\rho_{h_k^p(p)}^{-1}-\frac{\Vol(d\nu_{h_k(p)})}{n_p}\,
\right|_{\CC^0}\\=\rho_{h_k^p(p)}^{-1}
\frac{\Vol(d\nu_{h_k(p)})}{n_p}
\left|\,\rho_{h_k^p(p)}-\frac{n_p}{\Vol(d\nu_{h_k(p)})}\,
\right|_{\CC^0}
\leq C p^{-k-n}\,,
\end{multline}
so that by induction on the number $m\in\N$
of successive derivatives of $\rho_{h_k^p(p)}^{-1}$
and using \cref{approxbal} up to $m\in\N$, we get
constants $C_m>0$ such that for all $p\in\N$,
\begin{equation}\label{approxbalflainv}
\left|\,\rho_{h_k^p(p)}^{-1}-\frac{\Vol(d\nu_{h_k(p)})}{n_p}\,
\right|_{\CC^m}\leq C_m p^{-k-n}\,.
\end{equation}

On the other hand,
using the Sobolev embedding
theorem as in \cite[Lemma\,2]{MM15}, we get for any $m\in\N$
and $h\in\Met^+(L^p)$ a constant
$C_m>0$ such that for all $p\in\N$ and any holomorphic section
$s\in H^0(X,L^p)$, we have
\begin{equation}\label{Sobunif}
|s|_{\CC^m(h^p)}\leq C_m\,p^{\frac{n+m}{2}}\|s\|_{L^2(h^p)}\,,
\end{equation}
where $|\cdot|_{\CC^m(h^p)}$ denotes the $\CC^m$-norm
with respect to the Chern connection of $(L^p,h^p)$.
Using formula \cref{approxbalmet}, this
inequality readily extends to the approximately
balanced metrics $h_k^p(p)\in\Met^+(L^p)$.

Writing now $\s_k(p)=\{s_j\}_{j=1}^{n_p}$, 
\cref{sumprop,FSvar} show that for all
$A=(A_{jk})_{j,\,k=1}^{n_p}\in\Herm(\C^{n_p})$, we have
\begin{equation}\label{Rsig-R}
\sigma_{\s_k(p)}(A)=\rho_{h_k^p(p)}^{-1}
\sum_{j,\,k=1}^{n_p}A_{jk}\<s_k,s_j\>_{h_k^p(p)}\,.
\end{equation}
Then using the estimates \cref{npexp,Sobunif,approxbalflainv}
together with Cauchy-Schwarz inequality on the trace norm,
we get for all $m\in\N$
constants $C,\,C',\,C''>0$
such that
for all $A\in\Herm(\C^{n_p})$ and all $p\in\N$,
we have
\begin{equation}\label{sigmaG}
\begin{split}
&|\sigma_{\s_k(p)}(A)|_{\CC^m}
\leq\left|\rho_{h_k^p(p)}^{-1}\right|_{\CC^m}
\sum_{j,\,k=1}^{n_p}\left|A_{jk}\<s_k,s_j\>_{h_k(p)}
\right|_{\CC^m}\\
&\leq\left(\frac{\Vol(d\nu_{h_k(p)})}{n_p}
+Cp^{-n-k}\right)C'p^{n+\frac{m}{2}}
\|A\|_{tr}\,n_p\\
&\leq C'' p^{n+\frac{m}{2}}\|A\|_{tr}\,.
\end{split}
\end{equation}
This implies in particular that for all
$B\in\Herm(\C^{n_p})$ with $\|B\|_{tr}\leq C^{-1}p^{-k_0}$,
we have
\begin{equation}\label{sigmae2B-1}
\left|\sigma_{\s_k(p)}(e^{2B})-1\right|_{\CC^m}=\left|
\sigma_{\s_k(p)}(e^{2B}-\Id)\right|_{\CC^m}
\leq Cp^{n+\frac{m}{2}-k_0}\,.
\end{equation}
Now by formula \cref{heBsp=hkp} and classical
properties of the Kähler form
\cref{preq}, we have
\begin{equation}\label{ddbar}
\begin{split}
\om_{e^B\s_k(p)}=\om_{h_k(p)}
&-\frac{\sqrt{-1}}{2\pi p}\dbar\partial\log\sigma_{\s_k(p)}\left(e^{2B}\right)
-\frac{\sqrt{-1}}{2\pi p}\dbar\partial\log\rho_{h_k^p(p)}\\
=\om_{h_k(p)}
&-\frac{\sqrt{-1}}{2\pi p}\dbar\partial\log
\left(1+\sigma_{\s_k(p)}\left(e^{2B}-\Id\right)\right)\\
&-\frac{\sqrt{-1}}{2\pi p}\dbar\partial
\log\left(1+\left(\frac{\Vol(d\nu_{h_k(p)})}{n_p}\rho_{h_k^p(p)}
-1\right)\right)\,,
\end{split}
\end{equation}
which by \cref{approxbal} and formula \cref{sigmae2B-1}
implies that for any $k,\,k_0,\,m\in\N$
with $k_0>n+m/2$, there exists $C>0$ such
for all
$B\in\Herm(\C^{n_p})$ with $\|B\|_{tr}\leq C^{-1}p^{-k_0}$,
we have
\begin{equation}
\left|\om_{e^{B}\s_k(p)}-\om_{h_k(p)}\right|_{\CC^{m-2}}
\leq\frac{C}{p}\,.
\end{equation}
By formula \cref{approxbalmet} for $h_k(p)$ and the
corresponding formula for $\om_{h_k(p)}$ as in \cref{ddbar},
this implies the result.
\end{proof}

In the case $m=0$, the estimate \cref{sigmae2B-1} admits an
elementary improvement. In fact,
\cref{cohstateprojdef} together with Cauchy-Schwarz inequality
and the fact that $\|\Pi_\s\|_{tr}=1$ implies that
for any $\epsilon>0$ small enough,
there is $C>0$ such that for any $B\in\Herm(\C^{n_p})$
with $\|B\|_{tr}\leq\epsilon$
and any $\s\in\BB(H^0(X,L^p))$, we have
\begin{equation}\label{sigmae2B-12}
|\sigma_\s(e^{2B})-1|_{\CC^0}\leq C\|B\|_{tr}\,.
\end{equation}
This inequality will be used repeatedly in all the sequel.

One of the technical differences of our situation
compared to the classical situation of \cref{Liouvilleex}
is the fact that the volumes of the anticanonical volume
map depend on the positive Hermitian metric. To control
these volumes, we will use the following Lemma,
where for any $\s\in\BB(H^0(X,L^p))$,
we write $d\nu_\s$ for the anticanonical volume
form \cref{dnucandef} associated with $h_\s\in\Met^+(L)$.

\begin{lem}\label{vollem}
For any $k_0,\,k\in\N$ with $k\geq k_0$,
there exists a constant $C>0$ such that
for all $p\in\N$ and any
$B\in\Herm(\C^{n_p})$ with $\|B\|_{tr}\leq C^{-1}p^{-k_0}$,
we have
\begin{equation}\label{volborne}
\left|\frac{d\nu_{e^{B}\s_k(p)}}{d\nu_{h_k(p)}}-
\frac{\Vol(d\nu_{e^{B}\s_k(p)})}{\Vol(d\nu_{h_k(p)})}\,\right|_{\CC^0}
\leq Cp^{-k_0-1}\,,
\end{equation}
and $C^{-1}<\Vol(d\nu_{e^{B}\s_k(p)})<C$.
\end{lem}
\begin{proof}
Fix $k,\,k_0\in\N$ with $k\geq k_0$.
Using \cref{dnuef/dnu}, for any
$p\in\N$ and $B\in\Herm(\C^{n_p})$, formula \cref{heBsp=hkp}
gives
\begin{equation}\label{logvol}
\log\frac{d\nu_{e^B\s_k(p)}}{d\nu_{h_k(p)}}
=-\frac{1}{p}\log\rho_{h_k^p(p)}-\frac{1}{p}
\log\sigma_{\s_k(p)}(e^{2B})\,.
\end{equation}
Then using \cref{approxbal} and formula \cref{sigmae2B-12},
we get a constant $C>0$ such that
for all $p\in\N$ and all $B\in\Herm(\C^{n_p})$ with
$\|B\|_{tr}\leq C^{-1}p^{-k_0}$, we have
\begin{equation}
\begin{split}
&\left|\log\frac{d\nu_{e^B\s_k(p)}}{d\nu_{h_k(p)}}
-\frac{1}{p}\log\frac{\Vol(d\nu_{h_k(p)})}{n_p}\right|_{\CC^0}\\
&=\frac{1}{p}\left|\log\left(1+\left(\frac{\Vol(d\nu_{h_k(p)})}{n_p}\rho_{h_k^p(p)}-1\right)\right)+\log\left(1+\sigma_{\s_k(p)}\left(e^{2B}-\Id\right)\right)
\right|_{\CC^0}\\
&\leq Cp^{-k_0-1}\,.
\end{split}
\end{equation}
Note that we used the asymptotic expansion \cref{npexp}
for the dimension
and that $\Vol(d\nu_{h_k(p)})\to\Vol(d\nu_{h})$
as $p\to+\infty$, which also shows that
$\frac{1}{p}\log\frac{\Vol(d\nu_{h_k(p)})}{n_p}\to 0$
as $p\to+\infty$. In other words, there exist
constants $V_p>0$ satisfying $V_p\to 1$ as $p\to+\infty$
such that for all $B\in\Herm(\C^{n_p})$ with
$\|B\|_{tr}\leq C^{-1}p^{-k_0}$, we have
\begin{equation}
\left|\frac{d\nu_{e^B\s_k(p)}}
{d\nu_{h_k(p)}}-V_p\,\right|_{\CC^0}
\leq Cp^{-k_0-1}\,.
\end{equation}
Taking the integral of both sides against $d\nu_{h_k(p)}$,
we see that there is $C>0$ such that the constants $V_p>0$
for all $p\in\N$ satisfy
\begin{equation}
\left|V_p-\frac{\Vol(d\nu_{e^{B}\s_k(p)})}{\Vol(d\nu_{h_k(p)})}\right|<
Cp^{-k_0-1}\,.
\end{equation}
This gives the result.
\end{proof}

\subsection{Convergence of the balanced metrics}

In this Section, we consider a Fano manifold $X$ endowed with $L:=K_X^*$,
and work in the setting
the anticanonical volume map
$\nu:\Met(K_X^*)\longrightarrow\MM(X)$
defined by formula \cref{dnucandef}.

The goal of this section is to establish \cref{mainth}.
The proof is based on the following fundamental
link between the Berezin-Toeplitz
quantum channel of \cref{quantchandef} associated with an
anticanonically balanced metric
and the derivative of the moment map of \cref{momentdef}
at the corresponding anticanonically balanced product.
For any $\s\in\BB(H^0(X,L^p))$ and $A\in\Herm(\C^{n_p})$, write
\begin{equation}
D_\s\mu_\nu(A):=\dt\Big|_{t=0}\,\mu_\nu(e^{tA}\s)\,.
\end{equation}
To simplify notations,
we will write $d\nu_\s\in\MM(X)$ for the anticanonical volume
form \cref{dnucandef} associated with $h_\s\in\Met^+(L)$.

\begin{prop}\label{dmupre}
Assume that $h^p\in\Met^+(L^p)$ is balanced with respect to the
anticanonical volume form \cref{dnucandef}, and let
$\s_p\in\BB(H^0(X,L^p))$ be orthonormal with respect to $L^2(h^p)$.
Then for all $A\in\Herm(\C^{n_p})$ with $\Tr[A]=0$
and all $\s\in\BB(H^0(X,L^p))$, we have
\begin{equation}\label{TrAdmuA}
\frac{n_p}{2\Vol(d\nu_{\s_p})}\Tr[A\,D_{\s_p}\mu_\nu(A)]=\Tr[A^2]-\left(1+\frac{1}{p}\right)
\Tr[A\EE_{h^p}(A)]\,.
\end{equation}
\todo{Typo corrigée: $\EE_{h^p}(A)\mapsto A\EE_{h^p}(A)$}
\end{prop}
\begin{proof}
Let us first compute $D_{\textbf{s}}\mu_\nu(A)\in\Herm(\C^{n_p})$,
for general $\textbf{s}\in\BB(H^0(X,L^p))$ and
$A\in\Herm(\C^{n_p})$ with $\Tr[A]=0$.
First recall from \cref{FSvar} that
\begin{equation}
\dt\Big|_{t=0}\,h_{e^{tA}\textbf{s}}^p=
-2\sigma_\s(A)\,h_{\textbf{s}}^p\,.
\end{equation}
Recall also that $\Pi_\s:X\to\Herm(\C^{n_p})$
denotes the coherent state
projector of
\cref{cohstateprojdef} associated with $H_\s\in\Prod(H^0(X,L^p))$
under the identification \cref{isomHerm}
induced by any $\s\in\BB(H^0(X,L^p))$. Writing
$\s=:\{s_j\}_{j=1}^{n_p}$, \cref{FSdef} implies
that for all $x\in X$, we have
\begin{equation}\label{cohstatecoord}
\Pi_{\s}(x)=\Big(\<s_j(x),s_k(x)\>_{h_\s^p}
\Big)_{j,\,k=1}^{n_p}\,.
\end{equation}
Then by \cref{momentdef,dnuef/dnu}, we compute
\begin{multline}\label{dmucomput}
D_{\textbf{s}}\mu_\nu(A)=\left(\int_X\,\dt\Big|_{t=0}
\<e^{tA}s_j,e^{tA}s_k\>_{h_{e^{tA}\s}^p}
\,d\nu_{\textbf{s}}
\right)_{j,\,k=1}^{n_p}\\
+\left(\int_X\,
\<s_j,s_k\>_{h_\s^p}
\,\dt\Big|_{t=0}d\nu_{e^{tA}\textbf{s}}\right)_{j,\,k=1}^{n_p}
-\left(\dt\Big|_{t=0}\frac{\Vol(d\nu_{e^{tA}\textbf{s}})}{n_p}\right)\Id_{\C^{n_p}}\\
=\int_X (A\Pi_\s+\Pi_\s A-2\sigma_\s(A)\Pi_\s)\,d\nu_{\textbf{s}}-\frac{2}{p}\int_X\,
\sigma_\s(A)\Pi_\s\,d\nu_{\textbf{s}}\\
-\,
\left(\frac{2}{pn_p}\int_X\,\sigma_\s(A)\,d\nu_\s\right)\,\Id_{\C^{n_p}}\,,
\end{multline}
so that using \cref{cohstateprojdef} and the fact that
$\Tr[A]=0$, we get
\begin{equation}\label{dmucomputfinal}
\frac{1}{2}\Tr[A\,D_{\textbf{s}}\mu_\nu(A)]=\int_X\sigma_\s(A^2)
\,d\nu_\s-
\left(1+\frac{1}{p}\right)\int_X\sigma_\s(A)^2\,d\nu_\s\,.
\end{equation}
On the other hand, for any $h\in\Met(L^p)^+$ and letting
$\s_p\in\BB(H^0(X,L^p))$ be orthonormal with respect to $L^2(h^p)$,
by \cref{cohstateprojdef}, \cref{Rawnprop} and
formula \cref{quatnchantr} for the
quantum channel of \cref{quantchandef}, we have
\begin{equation}\label{quantchancomput}
\begin{split}
\Tr[A^2]&=\int_X \sigma_{\s_p}(A^2)\,\rho_{h^p}
\,d\nu_{h}\,,\\
\Tr\left[A\,\EE_{h^p}(A)\right]&=
\int_X\,\sigma_{\s_p}(A)^2
\,\rho_{h^p}\,d\nu_{h}\,.
\end{split}
\end{equation}
Then comparing formulas \cref{dmucomputfinal}
and \cref{quantchancomput} with
$h\in\Met(L^p)^+$ balanced with respect to
the anticanonical volume form \cref{dnucandef},
so that $h^p=h_{\s_p}$, and using
\cref{balmet=balemb}, we get the result.
\end{proof}

Consider now the setting of the previous Section, so
that $\Aut(X)$ is discrete and $X$
admits a Kähler-Einstein metric $h\in\Met^+(K_X^*)$.
For any $k\in\N$ and $p\in\N$ big enough,
let $\textbf{s}_k(p)\in
\BB(H^0(X,L^p))$ be orthonormal bases
for the $L^2$-products
$L^2(h_k^p(p))\in\Prod(H^0(X,L^p))$ induced by
\cref{approxbal}.
The key part of the proof of \cref{mainth}
is the following result, giving
a lower bound for the
derivative of the moment map at the approximately balanced
bases. 
It is based on the asymptotics of \cref{Bpgap} on the spectral
gap of the quantum channel $\EE_{h_k^p(p)}$ associated to
$h^p_k(p)$, which allow us to bypass the difficult geometric
argument in the proofs of
Donaldson \cite{Don01} and Phong and Sturm \cite{PS04}
of the analogous result for \cref{Liouvilleex}.

\begin{prop}\label{dmu}
For any $k,\,k_0\in\N$ with $k\geq k_0>n+1$,
there exists a constant $\epsilon>0$ such that for all
$p\in\N$ big enough, for all
$B\in\Herm(\C^{n_p})$ with $\|B\|_{tr}\leq \epsilon p^{-k_0}$
and all $A\in\Herm(\C^{n_p})$ with $\Tr[A]=0$, we have
\begin{equation}\label{TrAdmuA>A}
\frac{n_p}{\Vol(d\nu_{e^B\s_k(p)})}
\Tr[A\,D_{e^B\s_k(p)}\mu_\nu(A)]
\geq\frac{\epsilon}{p}\,\|A\|_{tr}^2\,.
\end{equation}
\end{prop}
\begin{proof}
The proof consists of an approximate version of \cref{dmupre}, whose
proof will be used in a crucial way.
We will use the following inequality, which holds
for any triple of Hermitian matrices
$A,\,B,\,G\in\Herm(\C^{n_p})$ as a consequence of
Cauchy-Schwarz inequality,
\begin{equation}\label{trid}
\left|\Tr[ABG]\right|\leq\|A\|_{tr}\|B\|_{tr}\|G\|_{op}\,.
\end{equation}
By \cref{cohstateprojdef} and the fact that
$\|\Pi_\s\|_{tr}=\|\Pi_\s\|_{op}=1$, this shows that
for all $A\in\Herm(\C^{n_p})$ and all
$\s\in\BB(H^0(X,L^p))$,
\begin{equation}\label{basicest}
|\sigma_\s(A)|_{\CC^0}\leq \|A\|_{tr}\quad~~\text{and}~~\quad
|\sigma_\s(A^2)|_{\CC^0}\leq\|A\|^2_{tr}\,.
\end{equation}
Using \cref{FSvar2}, the submultiplicativity of the operator
norm and the fact that $\|B\|_{op}\leq\|B\|_{tr}$ for all
$B\in\Herm(\C^{n_p})$, the inequality \cref{trid}
also shows that
that for any $\epsilon>0$,
there is a constant $C>0$
such that for all $B\in\Herm(\C^{n_p})$
with $\|B\|_{tr}\leq\epsilon p^{-k_0}$ and all $p\in\N$,
we have
\begin{equation}\label{sigmaeBest}
\begin{split}
|\sigma_{e^{B}\s}(A)^2-\sigma_\s(A)^2|_{\CC^0}
&\leq 2\|A\|_{tr}\left|
\sigma_{\s}(e^{2B})^{-1}
\sigma_{\s}(e^{B}Ae^{B})-\sigma_\s(A)\right|_{\CC^0}\\
&\leq Cp^{-k_0}\|A\|_{tr}^2\,,
\end{split}
\end{equation}
and in the same way,
\begin{equation}
\begin{split}
|\sigma_{e^{B}\s}(A^2)-\sigma_\s(A^2)|_{\CC^0}
&=\left|
\sigma_{\s}(e^{2B})^{-1}
\sigma_{\s}(e^{B}A^2e^{B})-\sigma_\s(A^2)\right|_{\CC^0}\\
&\leq Cp^{-k_0}\|A\|_{tr}^2\,.
\end{split}
\end{equation}
Consider the operator $S_p$ acting on
$A\in\Herm(\C^{n_p})$ by
\begin{equation}\label{Spdef}
S_p(A):=A-\left(1+\frac{1}{p}\right)\EE_{h_k^p(p)}(A)\,.
\end{equation}
Assume now $k\geq k_0>n$, and recall that
$\s_k(p)\in\BB(H^0(X,L^p))$ is an orthonormal
basis for $L^2(h^p_k(p))$, for all $p\in\N$ big enough.
Then plugging $\s=e^{B}\s_k(p)$ into
\cref{dmucomputfinal} and comparing with
\cref{quantchancomput} for $h_k(p)$, we can use
\cref{approxbal,vollem}
together with \cref{basicest,sigmaeBest,npexp},
to get a constant $C>0$
such that for all $p\in\N$ big enough,
for all $B\in\Herm(\C^{n_p})$
with $\|B\|_{tr}\leq C^{-1} p^{-k_0}$ and 
for all $A\in\Herm(\C^{n_p})$ with $\Tr[A]=0$,
we have
\begin{multline}\label{dmuestN}
\left|\frac{n_p}{2\Vol(d\nu_{e^B\s_k(p)})}
\Tr[A\,D_{e^B\s_k(p)}\,\mu_\nu(A)]-\Tr[A\,
S_p(A)]\right|\\
\leq\int_X\left|\sigma_{e^B\s_k(p)}(A^2)
\frac{n_p}{\Vol(d\nu_{e^B\s_k(p)})}
\frac{d\nu_{e^B\s_k(p)}}{d\nu_{h_k(p)}}-
\sigma_{\s_k(p)}(A^2)\rho_{h_k^p(p)}\right|\,
d\nu_{h_k(p)}\\
+\int_X\left|\sigma_{e^B\s_k(p)}(A)^2\frac{n_p}
{\Vol(d\nu_{e^B\s_k(p)})}\frac{d\nu_{e^B\s_k(p)}}{d\nu_{h_k(p)}}-
\sigma_{\s_k(p)}(A)^2\rho_{h_k^p(p)}\right|d\nu_{h_k(p)}\\
\leq C\,p^{n-k_0}\|A\|_{tr}^2\,.
\end{multline}
Recall that for any $h\in\Met^+(L)$, we write
$\lambda_1(h)>0$ for the first positive eigenvalue
of the Riemannian Laplacian of $(X,g^{TX}_h)$ acting
on $\cinf(X,\C)$.
Then formula \cref{approxbalmet} shows that
there exists
a constant $C>0$ such that for all $p\in\N$,
we have
\begin{equation}
|\lambda_1(h_k(p))-\lambda_1(h_\infty)|\leq C/p\,.
\end{equation}
Using the uniformity
in \cref{Bpgap},
this gives a constant $C>0$ such that for all $p\in\N$,
\begin{equation}\label{Spest}
\begin{split}
\Tr[A\,S_p(A)]
&\geq\|A\|_{tr}^2-\left(1+\frac{1}{p}\right)\left(1-
\frac{\lambda_1(h_\infty)}{4\pi p}
+Cp^{-2}\right)\|A\|_{tr}^2\\
&\geq\left(\frac{\lambda_1(h_\infty)-4\pi}{4\pi p}-Cp^{-2}
\left(1+\frac{1}{p}\right)\right)
\|A\|^2_{tr}\,.
\end{split}
\end{equation}
Using \cref{lambda1>4pi} and assuming $k\geq k_0>n+1$,
we get from the estimates \cref{dmuestN} and \cref{Spest}
a constant $\epsilon>0$ such that for all
$p\in\N$ big enough, for all $B\in\Herm(\C^{n_p})$
with $\|B\|_{tr}<\epsilon p^{-k_0}$
and all $A\in\Herm(\C^n)$ with $\Tr[A]=0$, we have
\begin{equation}
\frac{n_p}{\Vol(d\nu_{e^B\s_k(p)})}\Tr[A\,D_{e^B\s_k(p)}
\,\mu_\nu(A)]
\geq\frac{\epsilon}{p}\|A\|_{tr}^2\,.
\end{equation}
This gives the result.
\end{proof}

In the following result, we show that
the moment map Lemma of Donaldson
in \cite[Prop.\,17]{Don01} is valid in our setting,
although we do not exhibit any associated Kähler structure.

\begin{prop}\label{Donlem}
Fix $p\in\N$ and assume that there exist $\s\in\BB(H^0(X,L^p))$
and $\lambda,\,\delta>0$ such that
\begin{itemize}
\item[$(1)$] $\lambda\,\|\mu_\nu(\s)\|_{tr}<\delta\,;$
\item[$(2)$] $\lambda\Tr[A\,D_{e^{B}\s}\mu_\nu(A)]\geq \|A\|^2_{tr}
~\text{for all}~A\in\Herm(\C^{n_p})~\text{such that}~\Tr[A]=0\\
\text{and all}~B\in\Herm(\C^{n_p})~\text{such that}~
\|B\|_{tr}\leq\delta$.
\end{itemize}
Then there exists $B\in\Herm(\C^{n_p})$
with $\|B\|_{tr}\leq\delta$
and $\mu_\nu(e^{B}\s)=0$.
\end{prop}
\begin{proof}
First note that for any unitary endomorphism
$U\in U(n_p)$ and any $\s\in\BB(H^0(X,L^p))$,
\cref{momentdef} shows that $\mu_\nu(U\s)=U\mu_\nu(\s)\,U^*$.
Thus for any $A,\,B\in\Herm(\C^{n_p})$, one
computes that
\begin{equation}
\begin{split}
\Tr[A\,D_{U\s}\mu_\nu(A)]
&=\dt\Big|_{t=0}\Tr[A\,\mu_\nu(e^{tA}U\s)]\\
&=\Tr[U^*AU\,D_\s\mu_\nu(U^*AU)]\,.
\end{split}
\end{equation}
In particular, assumption (2) is equivalent to
\begin{itemize}
\item[$(2')$] $\lambda\Tr[A\,D_{Ue^{B}\s}\mu_\nu(A)]\geq \|A\|^2_{tr}
~\text{for all}~A\in\Herm(\C^{n_p})~\text{such that}~\Tr[A]=0,\\
\text{all}~U\in U(n_p)~
\text{and all}~B\in\Herm(\C^{n_p})~\text{such that}~
\|B\|_{tr}\leq\delta$.
\end{itemize}
Let us now consider
$\mu_\nu:\BB(H^0(X,L^p))\to\Herm(\C^{n_p})$
as a vector field
on $\BB(H^0(X,L^p))$ via the identification
\cref{isomEnd}.
Let $\s\in\BB(H^0(X,L^p))$ be such that
assumptions (1) and (2) are satisfied, and let
$\s_t\in\BB(H^0(X,L^p))$ for all $t>0$ be the solution of the ODE
\begin{equation}\label{stdef}
\left\{
\begin{array}{l}
  \dt\,\s_t=-\mu_\nu(\s_t)\quad\text{for all}\quad t\geq 0\,, \\
  \\
  \s_0=\s\,.
\end{array}
\right.
\end{equation}
If $\mu_\nu(\s)=0$, then the result is trivially satisfied,
so that we can assume $\mu_\nu(\s)\neq 0$, in which case
$\mu_\nu(\s_t)\neq 0$ for all $t\geq 0$.
Let $t_0\geq 0$ be such that
there exist $U_t\in U(n_p)$ and $B_t\in\Herm(\C^{n_p})$
with $\|B_t\|_{tr}\leq\delta$ such that $\s_t=U_te^{B_t}\s$
for all $t\in[0,t_0]$.
Using assumption
$(2')$ and recalling that $\Tr[\mu_\nu]=0$ by \cref{sumprop},
for all $t\in[0,t_0]$ we have
\begin{equation}
-\lambda\dt\|\mu_\nu(\s_t)\|^2_{tr}=
2\lambda\Tr[\mu_\nu(\s_t)\,D_{\s_t}\mu_\nu(\mu_\nu(\s_t))]\geq
2\|\mu_\nu(\s_t)\|^2_{tr}\,.
\end{equation}
By derivation of the square, this implies
$\lambda\dt\|\mu_\nu(\s_t)\|_{tr}
\leq-\|\mu_\nu(\s_t)\|_{tr}$ for all
$t\in[0,t_0]$, so that using Grönwall's lemma with initial
condition (1) and the fact that
$\mu_\nu(\s_t)=U_t\mu_\nu(e^{B_t}\s)\,U^*_t$, we get
\begin{equation}\label{Gronwallappli}
\|\mu_\nu(e^{B_t}\s)\|_{tr}=
\|\mu_\nu(\s_t)\|_{tr}\leq e^{-t/\lambda}\,\|\mu_\nu(\s_0)\|_{tr}<
\frac{\delta}{\lambda}\,e^{-t/\lambda}\,.
\end{equation}
Let us now consider $\Prod(H^0(X,L^p))$ as a symmetric
space via the quotient map \cref{symmap}, and
recall that the geodesics are the image of the $1$-parameter
groups of the action of $\GL(\C^{n_p})$ as in formula
\cref{geod}.
Then by equation \cref{stdef}, the
Riemannian length $L(t_0)\geq 0$
of the path
$\{t\mapsto H_{\s_t}\}_{t\in[0,t_0]}\subset\Prod(H^0(X,L^p))$ satisfies
\begin{equation}\label{path<delta}
L(t_0)=\int_0^{t_0}\,\|\mu_\nu(\s_t)\|_{tr}\,dt
<\frac{\delta}{\lambda}\int_0^{+\infty}
e^{-t/\lambda}\,dt=\delta\,.
\end{equation}
This means that there exists $\epsilon>0$
such that all points of
$\{t\mapsto H_{\s_t}\}_{t\in[0,t_0+\epsilon]}$
can be joined by a geodesic of
length strictly less than $\delta$, i.e., that for each
$t\in[0,t_0+\epsilon]$, there exists $B_t\in\Herm(\C^{n_p})$
with $\|B_t\|_{tr}\leq\delta$ such that
$H_{\s_t}=H_{e^{B_t}\s}$, so that there exists $U_t\in U(n_p)$
such that $\s_t=U_te^{B_t}\s$.
Thus $I:=\{t_0\geq 0\,|\,L(t_0)<\delta\}$
is non-empty, open and closed in $[0,+\infty[$,
so that $I=[0,+\infty[$.
In particular, the path $\{t\mapsto H_{\s_t}\}_{t>0}$
has total Riemannian length strictly less than $\delta$,
so that it converges to a limit point 
$H_{e^{B_\infty}\s}\in\Prod(H^0(X,L^p))$ by completeness,
with
$B_\infty\in\Herm(\C^{n_p})$ satisfying
$\|B_\infty\|_{tr}\leq\delta$.
Finally, inequality \cref{Gronwallappli}
for all $t>0$
implies
\begin{equation}
\|\mu_\nu(e^{B_\infty}\s)\|_{tr}=
\lim_{t\fl+\infty}\|\mu_\nu(e^{B_t}\s)\|_{tr}=0\,.
\end{equation}
This gives the result.
\end{proof}

With all these prerequisites in hand, we can finally give the
proof of \cref{mainth}.
\newline

{\noindent
\textbf{Proof of \cref{mainth}.}}
First note by
\cref{Rawnprop} and formula \cref{cohstatecoord}
that for any
$k\in\N$ and $p\in\N$ big enough,
the value of the
moment map of \cref{momentdef}
at
the orthonormal basis $\s_k(p)\in\BB(H^0(X,L^p))$
for $L^2(h_k^p(p))$ satisfies
the following formula ,
\begin{equation}\label{muRawn}
\frac{n_p}{\Vol(d\nu_{\s_k(p)})}\mu_\nu(\s_k(p))=\int_X
\Pi_{\s_k(p)}\,\left(\frac{n_p}{\Vol(d\nu_{\s_k(p)})}
\frac{d\nu_{\s_k(p)}}{d\nu_{h_k(p)}}-
\rho_{h_k^p(p)}\right)
\,d\nu_{h_k(p)}\,.
\end{equation}
Thus using \cref{approxbal,vollem}
together with the estimate
\cref{npexp} for the dimension and the fact that
$\|\Pi_\s\|_{tr}=1$ for all $\s\in\BB(H^0(X,L^p))$,
we get a constant $C>0$ such that for all $p\in\N$, we have
\begin{equation}
\begin{split}
\frac{n_p}{\Vol(d\nu_{\s_k(p)})}&\|\mu_\nu(\s_k(p))\|_{tr}\\
&\leq\Vol(d\nu_{h_k(p)})\,
\left|\,\frac{n_p}{\Vol(d\nu_{\s_k(p)})}
\frac{d\nu_{\s_k(p)}}{d\nu_{h_k(p)}}
-\rho_{h_k^p(p)}\,\right|_{\CC^0}\leq C p^{n-k}\,.
\end{split}
\end{equation}
Thus taking $k_0>n+1$, we can then choose $k>k_0+n+1$,
and \cref{dmu}
shows that 
\cref{Donlem} applies
for $p\in\N$ big enough and $\s=\s_k(p)$, with
\begin{equation}
\lambda=\frac{p}{\epsilon}\frac{n_p}{\Vol(d\nu_{\s_k(p)})}
~~~~~\text{and}~~~~~
\delta=\frac{C}{\epsilon}p^{n+1-k}\,.
\end{equation}
This gives a sequence of Hermitian endomorphisms
$B_p\in\Herm(\C^{n_p})$, $p\in\N$, with 
$\|B_p\|_{tr}\leq \epsilon p^{-k_0}$ such that
$\mu_\nu(e^{B_p}\s_k(p))=0$ for all $p\in\N$ big enough.
By \cref{momentbal},
the Hermitian metrics $h_p:=h_{e^B\s_k(p)}^p\in\Met^+(L^p)$
are then anticanonically balanced
for all $p\in\N$ big enough,
and the associated Kähler forms satisfy
\begin{equation}
\om_{h_p}=p\,\om_{e^B\s_k(p)}\,,
\end{equation}
where $\om_{e^B\s_k(p)}$ is induced by $h_{e^B\s_k(p)}$.
If we also chose $k_0> n+1+m/2$ for some
$m\in\N$, \cref{approxomkplem} shows
the $\CC^{m}$-convergence \cref{mainthfla} to the Kähler-Einstein
form $\om_\infty$. This establishes \cref{mainth}.

\qed

\section{Donaldson's iterations towards anticanonically balanced metrics}
\label{donsec}


In this Section, we consider a Fano manifold $X$, together
with its anticanonical line bundle $L:=K_X^*$
and the associated anticanonical
volume map \cref{dnucandef}.
We will apply \cref{mainth}
to establish the exponential
convergence of the associated Donaldson's iterations
and compute the optimal rate of convergence.

\subsection{Donaldson map}

Our goal is to study the following dynamical system on
the space $\Prod(H^0(X,L^p))$ of Hermitian inner products on
$H^0(X,L^p)$. To this end, recall \cref{FSdef} for the
Fubini-Study map $\FS:\Prod(H^0(X,L^p))\to\Met^+(L^p)$.

\begin{defi}\label{Tdef}
For any $p\in\N$ big enough, the associated
anticanonical \emph{Donaldson map}
is defined by
\begin{equation}\label{Tdeffla}
\TT_\nu:=\Hilb_\nu\circ\,\FS:\Prod(H^0(X,L^p))\longrightarrow
\Prod(H^0(X,L^p))\,,
\end{equation}
where $\Hilb_\nu:\Met^+(L^p)\to\Prod(H^0(X,L^p))$ is the
\emph{anticanonical Hilbert map} defined by \cref{L2intro} using
the anticanonical
volume form \cref{dnucandef}.
\end{defi}

By construction, the balanced products of \cref{baldef} coincide
with the fixed points of the Donaldson map.
Using formula \cref{hFSdef} for the Fubini-Study metric
and writing $h_H^p:=\FS(H)\in\Met^+(L^p)$
for any $H\in\Prod(H^0(X,L^p))$,
we get the explicit
description
\begin{equation}\label{Tfla}
\TT_\nu(H)=\frac{n_p}{\Vol(d\nu_{h_H})}\,\int_{X}
\<\Pi_H(x)\,\cdot\,,\,\cdot\,\>_{H}
\,d\nu_{h_H}(x)\,,
\end{equation}
For any $h^p\in\Met^+(L^p)$ and $H\in\Prod(H^0(X,L^p))$,
consider the natural identifications
\begin{equation}\label{Metid}
\begin{split}
\cinf(X,\R)&\xrightarrow{~\sim~}T_{h^p}\Met^+(L^p)\\
f~&\longmapsto~\dt\Big|_{t=0}\,e^{tf}\,h^p\,,
\end{split}
\end{equation}
and
\begin{equation}\label{Prodid}
\begin{split}
\cL(H^0(X,L^p),H)&\xrightarrow{~\sim~}T_H\Prod(H^0(X,L^p))\\
A~&\longmapsto~\dt\Big|_{t=0}\,\<e^{tA}\cdot,\cdot\>_H\,.
\end{split}
\end{equation}
In the notations of \cref{momentsec}, if
$s\in\BB(H^0(X,L^p))$ is such that $H=H_\s$, then
for any $A\in\cL(H^0(X,L^p),H)$ we have
\begin{equation}\label{diffid}
H_{e^{A}\s}=H(e^{-2A}\,\cdot\,,\,\cdot\,)\,.
\end{equation}
In particular, the identification \cref{Prodid} differs
from the identification \cref{isomHerm} induced by the quotient map
\cref{symmap} by a factor of $-2$.

Recall now \cref{cohstateprojdef,BTquantdef}.

\begin{prop}\label{dHilb}
The derivative of the anticanonical
Hilbert map
at $h^p\in\Met^+(L^p)$ 
is given by
\begin{equation}\label{dHilbfla}
\begin{split}
D_{h^p}\Hilb_\nu:\cinf(X,\R)&\longrightarrow
\cL(H^0(X,L^p),\Hilb_\nu(h^p))\,,\\
f\,&\longmapsto\,\left(1+\frac{1}{p}\right)T_{h^p}(f)-\frac{1}
{p}
\left(\int_X\,f\,\frac{d\nu_h}{\Vol(d\nu_{h})}\right)\,\Id\,.
\end{split}
\end{equation}
The derivative of the Fubini-Study map
at $H\in\Prod(H^0(X,L^p))$ is given by
\begin{equation}\label{DFSfla}
\begin{split}
D_H\FS:\cL(H^0(X,L^p),H)&\longrightarrow\cinf(X,\R)\,,\\
A\,&\longmapsto\,\sigma_H(A)\,.
\end{split}
\end{equation}
\end{prop}
\begin{proof}
For any $f\in\cinf(X,\R)$ and $t\in\R$, set
\begin{equation}
h_t^p:=e^{tf}\,h^p\in\Met^+(L^p)\,.
\end{equation}
Then for any $s_1,\,s_2\in H^0(X,L^p)$, using \cref{dnuef/dnu}
and the fact that $T_{h^p}(1)=\Id$ by formula
\cref{Tpf}, we compute
\begin{equation}
\begin{split}
&\dt\Big|_{t=0}\<s_1,s_2\>_{\Hilb_\nu(h_t^p)}\\
&=\frac{n_p}{\Vol(d\nu_{h})}\left(\int_X\,\dt\Big|_{t=0}
\<s_1,s_2\>_{h^p_t}\,
d\nu_h+
\int_X\,\<s_1,s_2\>_{h^p}\,\dt\Big|_{t=0}d\nu_{h_t}\right)\\
&+\left(\dt\Big|_{t=0}\frac{n_p}{\Vol(d\nu_{h_t})}\right)
\int_X\,
\<s_1,s_2\>_{h_t^p}\,d\nu_h\\
&=\frac{n_p}{\Vol(d\nu_{h^p})}\int_X
\left(f+\frac{1}{p}f
-\frac{1}{p}\int_X\,f\,\frac{d\nu_h}{\Vol(d\nu_h)}\right)\,\<s_1,s_2\>_{h^p}\,d\nu_h\\
&=\frac{n_p}{\Vol(d\nu_{h^p})}
\left\langle \left(\left(1+\frac{1}{p}\right)T_{h^p}(f)
-\frac{1}{p}\int_X\,f\,\frac{d\nu_h}{\Vol(d\nu_h)}\right)
s_1,s_2\right\rangle_{L^2(h^p)}\,.
\end{split}
\end{equation}
Using formula \cref{Hilb=L2}, this
proves the first statement \cref{dHilbfla}.

On the other hand,
in the identifications \cref{Metid}, \cref{Prodid}
and using formula \cref{diffid}, the second statement
\cref{DFSfla} is a consequence of \cref{FSvar}.
\end{proof}
%

Let $H\in\Prod(H^0(X,L^p))$ be an anticanonically balanced
product, and
consider the setting of \cref{BTsec}
for the anticanonically balanced metric
$h^p_H:=\FS(H)\in\Met^+(L^p)$. Then in particular,
\cref{baldef} of a balanced product implies that
$\cL(H^0(X,L^p),H)$ and $\cL(\HH_p)$ coincide
as real Hilbert spaces.
Via the identification \cref{Prodid},
\cref{dHilb} implies the following formula for
the derivative of the anticanonical
Donaldson map at a fixed point.

\begin{cor}\label{dTT}
The differential of the anticanonical Donaldson map
at a fixed point $H\in\Prod(H^0(X,L^p))$
is given by the following
formula, for all $A\in\cL(\HH_p)$,
\begin{equation}\label{DTT=TT*}
D_H\TT_\nu(A)=\left(1+\frac{1}{p}\right)\EE_{h^p_H}(A)
-\frac{1}{p}\frac{\Tr[A]}{n_p}\,\Id_{\HH_p}\,.
\end{equation}
\end{cor}
\begin{proof}
\cref{baldef} of a balanced product implies that
$H$ coincides with $L^2(h^p_H)$ up to a multiplicative constant,
and \cref{cohstateprojdef}
then shows that
the Berezin symbol maps $\sigma_{L^2(h^p_H)}$
and $\sigma_H$ coincide.
Using \cref{Rawnprop,balmet=balemb}
for the balanced metric $h^p_H$, we then get
for all $A\in\cL(\HH_p)$,
\begin{equation}
\frac{n_p}{\Vol(d\nu_{h_H})}\left(\int_X\,\sigma_H(A)\,
d\nu_{h_H}\right)=\Tr[A]\,.
\end{equation}
Then using \cref{quantchandef,Tdef},
the result follows from \cref{dHilb}
and formula \cref{DFSfla}.
\end{proof}

\subsection{Energy functional}

In this Section, we consider a Fano manifold $X$
with $\Aut(X)$ is discrete, and show
that if the anticanonical Donaldson map of \cref{Tdef}
admits a fixed point, then  its
iterations converge to
this fixed point, which is unique up to a
multiplicative constant. The results in this
Section are essentially a combination of results
of  Berman \cite{Ber13} and
Berman, Boucksom, Guedj and Zeriahi
\cite{BBGZ13}.
We gather them here as they play a central
role in the proof of \cref{gapth} given in the
next Section.

Recall that we write $L:=K_X^*$ for the anticanonical
line bundle of $X$, and let us introduce the energy functional
$E:\Met^+(L^p)\to\R$ defined for any $h^p\in\Met^+(L^p)$ by
\begin{equation}
E(h^p):=-\log\Vol(d\nu_h)\,.
\end{equation}
It has been considered in \cite[\S\,6.3]{BBGZ13}
as a replacement of
the \emph{Aubin-Yau functional} in the anticanonical setting.
Its key property in our context
is the following Lemma of Berman \cite[Lemma 2.6]{Ber13},
for which we give a proof
as it is quite elementary.

\begin{lem}\label{Lfctlem}
For any $h^p\in\Met^+(L^p)$, we have
\begin{equation}
E(\FS\circ\Hilb_\nu(h^p))\leq E(h^p)\,.
\end{equation}
\end{lem}
\begin{proof}
Let us first show that $E:\Met^+(L^p)\to\R$
is concave along paths in $\Met^+(L^p)$ of the
form
\begin{equation}\label{htdef}
t\longmapsto h_t^p:=e^{-tf}h^p\,,~t\in\R\,,
\end{equation}
for any $f\in\cinf(X,\R)$ such that $e^{-f}h^p$ is positive.
By \cref{dnuef/dnu}, for any $t\in\R$
we have
\begin{equation}\label{dtL}
\dt\,E(h_t^p)=\int_X\,f\,\frac{d\nu_{h_t}}{\Vol(d\nu_{h_t})}\,,
\end{equation}
so that using the Cauchy-Schwarz inequality, we get
\begin{equation}\label{dt2L}
\frac{\partial^2}{\partial t^2}\,E(h_t^p)=\left(\int_X\,f\,
\frac{d\nu_{h_t}}{\Vol(d\nu_{h_t})}
\right)^2-\int_X\,f^2\,\frac{d\nu_{h_t}}{\Vol(d\nu_{h_t})}
\leq 0\,.
\end{equation}
Recall the setting of \cref{BTsec} for $h^p\in\Met^+(L^p)$,
and let us take
\begin{equation}\label{varphidef}
f:=\frac{1}{p}\log\left(\frac{\Vol(d\nu_{h})}{n_p}\,\rho_{h^p}
\right)\,,
\end{equation}
so that $h_1^p=\FS\circ\Hilb_\nu(h^p)$ by
\cref{Rawnprop},
and consider the smooth function $\Phi:\R\to\R$
defined for any
$t\in\R$ by
\begin{equation}
\Phi(t):=E(h_t^p)-E(h_0^p)-t\int_X\,f\,\frac{d\nu_{h}}
{\Vol(d\nu_h)}\,.
\end{equation}
Then this function satisfies $f(0)=f'(0)=0$ by formula \cref{dtL} and is concave by formula \cref{dt2L},
so that in particular $f(1)\leq 0$ and
\begin{equation}
E(\FS\circ\Hilb_\nu(h^p))-E(h^p)
\leq \int_X\,f\,\frac{d\nu_h}{\Vol(d\nu_h)}
\,.
\end{equation}
Now using formula \cref{npexp} for the dimension,
the concavity of the logarithm implies
\begin{equation}
\int_X\,f\,\frac{d\nu_h}{\Vol(d\nu_h)}\leq
\frac{1}{p}
\log\left(\frac{1}{n_p}\int_X\,\rho_{h^p}\,d\nu_h\right)=0\,.
\end{equation}
This shows the result.
\end{proof}
%

From now on, we fix
a base point $H_0\in\Prod(H^0(X,L^p))$, and
identify any $H\in\Prod(H^0(X,L^p))$ with
a Hermitian endomorphism $H\in\cL(H^0(X,L^p),H_0)$ via
the formula 
\begin{equation}\label{basept}
H=H_0(H\cdot,\cdot)\,.
\end{equation}
Recall that $\Prod(H^0(X,L^p))$ is endowed with a natural
structure of a symmetric space via the quotient map
\cref{symmap}, with geodesics given by formula \cref{geod}.
The following result is a consequence of the
results of \cite{Ber09a,Ber09b} on positivity of direct images.

\begin{prop}\label{BBGZ}
{\cite[Lemma 7.2]{BBGZ13}}
Assume that $\Aut(X)$ is discrete. Then
the functional $\Psi:\Prod(H^0(X,L^p))\to\R$
defined for all $H\in\Prod(H^0(X,L^p))$ by
\begin{equation}\label{Psidef}
\Psi(H)=E(\FS(H))+\frac{1}{p}\frac{\log\det H}{n_p}
\end{equation}
is convex along geodesics of $\Prod(H^0(X,L^p))$, and
strictly convex when the geodesic is not
generated by a multiple of the identity.
\end{prop}

The fundamental role of the energy functional \cref{Psidef}
in finding anticanonically balanced products comes from the
following identity, which follows from \cref{FSvar} as in
the proof of \cref{dmupre} for all $\s\in\BB(H^0(X,L^p))$
and all $A\in\Herm(\C^{n_p})$,
\begin{equation}\label{dpsi=mu}
\frac{d}{dt}\Big|_{t=0}\Psi(H_{e^{tA}\s})=-\frac{2}{p\,\Vol(d\nu_\s)}\Tr[\mu(\s)A]\,.
\end{equation}
By \cref{momentbal}, this implies in particular that
critical points of $\Psi$ coincide with
anticanonically balanced products, and \cref{BBGZ} shows that
they are unique up to a multiplicative constant.
This also implies the following
result on the iterations of Donaldson's map,
due to Berman \cite[Th.\,4.14]{Ber13}. It 
essentially follows the proof of Donaldson
in \cite[Prop.\,4]{Don09}, and we give it here as it will be
used in the next Section.

\begin{prop}\label{TTpcvth}
Assume that $\Aut(X)$ is discrete, and let $p\in\N$
be such that an anticanonically balanced
product exists. Then for any
$H_0\in\Prod(H^0(X,L^p))$, there exists 
an anticanonically balanced product $H\in\Prod(H^0(X,L^p))$
such that
\begin{equation}\label{TTpcvfla}
\TT_\nu^k(H_0)\xrightarrow{k\fl+\infty}H\,.
\end{equation}
\end{prop}
\begin{proof}
Let us first show that for any $H\in\Prod(H^0(X,L^p))$,
we have $\Psi\left(\TT_\nu(H)\right)\leq\Psi(H)$.
By formula \cref{Tfla} and via the identification
\cref{basept}, as
$\Pi_H$ is rank-$1$ we have
\begin{equation}\label{logdet}
\Tr\left[\frac{\TT_\nu(H)H^{-1}}{n_p}\right]
=\frac{1}{\Vol(d\nu_{h_H})}\,\int_X\,\Tr[\Pi_H]\,
d\nu_{h_H}=1\,,
\end{equation}
so that
by concavity of the logarithm,
\begin{equation}\label{logconc}
\begin{split}
\frac{\log\det \TT_\nu(H)}{n_p}
-\frac{\log\det H}{n_p}&=\frac{\log\det\left(\TT_\nu(H)H^{-1}
\right)}{n_p}\\
&\leq
\log\Tr\left[\frac{\TT_\nu(H)H^{-1}}{n_p}\right]=0\,,
\end{split}
\end{equation}
with equality if and only if $\TT_\nu(H)=H$.
On the other hand, using \cref{Lfctlem} we get
$E(\TT_\nu(\FS(H)))\leq E(\FS(H))$, so that
$\Psi\left(\TT_\nu(H)\right)\leq\Psi(H)$, for all
$H\in\Prod(H^0(X,L^p))$.

Now by \cref{BBGZ} and identity \cref{dpsi=mu}, the existence of a 
balanced product implies that the functional $\Psi$
is bounded from below, so that in particular,
the decreasing sequence
$\{\Psi(\TT_\nu^r(H_0))\}_{r\in\N}$ converges to its lower bound.
As the Donaldson map $\TT_\nu$
decreases both terms of \cref{Psidef} separately,
this implies that the decreasing sequence
$\{\log\det(\TT_\nu^r(H_0))\}_{r\in\N}$ is also bounded from
below, so that $\{\det(\TT_\nu^r(H_0))\}_{r\in\N}$ is bounded in 
$]0,+\infty[$ and
\begin{equation}\label{logdetto0}
\frac{1}{n}\log\det\left(
\TT_\nu^{r+1}(H_0)\TT_\nu^r(H_0)^{-1}\right)
\longrightarrow 0\quad\text{as}~~r\to+\infty\;.
\end{equation}
Again by \cref{BBGZ} and identity \cref{dpsi=mu}, the existence of a 
balanced product implies that the functional $\Psi$
is proper over any subset of $\Prod(H^0(X,L^p))$
with bounded determinant.
We thus get that the sequence $\{\TT_\nu^r(H_0)\}_{r\in\N}$
admits an accumulation point
$H_p\in\Prod(H^0(X,L^p))$.
On the other hand, the equality case in
formula \cref{logdet} and formula \cref{logdetto0}
implies
\begin{equation}
\TT_\nu^{r+1}(H_0)\TT_\nu^r(H_0)^{-1}
\longrightarrow\Id,\quad\text{as}~~r\to+\infty\;.
\end{equation}
We thus get that $H\in\Prod(H^0(X,L^p))$ is the unique
accumulation point, and satisfies
$\TT_\nu(H_p)=H$. This concludes the proof.
\end{proof}

\subsection{Exponential convergence of Donaldson's iterations}

This Section is dedicated to the proof of \cref{gapth}.
It follows the argument of the analogous result
in \cite[Th.\,4.4]{IKPS19}
for the constant
volume map of \cref{cstantex}.

Consider the setting of \cref{BTsec} for an anticanonically
balanced metric $h^p\in\Met^+(L^p)$, so that
$H:=L^2(h^p)\in\Prod(H^0(X,L^p))$
is an anticanonically balanced product.

Recall that if $H\in\Prod(H^0(X,L^p))$
is an anticanonically balanced product,
then we have $\cL(H^0(X,L^p),H)=\cL(\HH_p)$ as real Hilbert spaces
for the
trace norm. Write
$D_H\TT_\nu:\cL(\HH_p)\to\cL(\HH_p)$ for
the differential
of the Donaldson map at $H$ in the
identification \cref{Prodid}.
%

\begin{lem}\label{expcvlem}
Let $X$ be a Fano manifold with $\Aut(X)$ discrete 
admitting a polarized Kähler-Einstein metric, and
let $\{H_p\in\Prod(H^0(X,L^p))\}_{p\in\N}$ be a
sequence of anticanonically
balanced products for all $p\in\N$ big enough.

Then $D_{H_p}\TT_\nu$
is an injective self-adjoint operator acting on $\cL(\HH_p)$
satisfying $D_{H_p}\TT_\nu(\Id)=\Id$.
Furthermore, the highest eigenvalue
$\gamma_1(H_p)\in\R$ of its restriction to
the subspace of traceless
matrices satisfies the following estimate as $p\to+\infty$,
\begin{equation}\label{gamma1est}
\gamma_1(H_p)= 1-\frac{\lambda_1
-4\pi}{4\pi p}+O(p^{-2})\,,
\end{equation}
where $\lambda_1>0$ is the first positive
eigenvalue of the Riemannian Laplacian associated
with the polarized Kähler-Einstein metric acting on
$\cinf(X,\C)$.
\end{lem}
\begin{proof}
Recall from \cref{quantchanprop} that the 
quantum channel of \cref{quantchandef}
is a self-adjoint operator acting on $\cL(\HH_p)$,
so that by \cref{dTT}, the differential
$D_{H_p}\TT_\nu$ is self-adjoint and satisfies
$D_{H_p}\TT_\nu(\Id)=\Id$.
In particular, it preserves the orthogonal of the identity, i.e.,
the space of traceless endomorphisms,
and \cref{dTT} implies that
for all $A\in\cL(\HH_p)$ with $\Tr[A]=0$,
we have
\begin{equation}
D_{H_p}\TT_\nu(A)=\left(1+\frac{1}{p}\right)\EE_{\FS(H_p)}(A)\,.
\end{equation}
Then \cref{quantchanprop} implies that $D_{H_p}\TT_\nu$
is injective and positive as an operator acting on $\cL(\HH_p)$.

To establish formula \cref{gamma1est}, recall from
\cref{BBGZ} and identity \cref{dpsi=mu} that if
$\Prod(H^0(X,L^p))$ contains an
anticanonically balanced product, then it
is unique up to a multiplicative constant. Furthermore, \cref{Tdef}
shows that $\TT_\nu(cH_p)=c\TT_\nu(H_p)$ for every $c>0$,
so that the spectrum of $D_{H_p}\TT_\nu$ does not depend
on the chosen anticanonically balanced product.
Using \cref{mainth},
to compute the estimate \cref{gamma1est}, we can then assume
that $H_p:=L^2(h(p))$ for each $p\in\N$,
where $\{h(p)\in\Met^+(L)\}_{p\in\N}$
is a sequence of positive Hermitian metrics
converging to the Kähler-Einstein metric
$h_\infty\in\Met^+(L)$.
The statement is then
an immediate
consequence of
the uniformity in \cref{Bpgap}, as in
the proof of \cref{dmu}.
\end{proof}

Recall now that $\Prod(H^0(X,L^p))$ admits
a natural structure of a symmetric space via the quotient
map \cref{symmap}, and write
$\textup{dist}(\cdot,\cdot)$ for the associated distance.
Using \cref{expcvlem} and the geometric input of the previous
Section, we can now give the proof of \cref{gapth}
following \cite[Th.\,4.4]{IKPS19}.
\newline

{\noindent
\textbf{Proof of \cref{gapth}.}}
Fix $p\in\N$ such that an anticanonically balanced
product exists by \cref{mainth},
and fix any $H_0\in\Prod(H^0(X,L^p))$.
By \cref{TTpcvth}, there exists an anticanonically
balanced product $H_p\in\Prod(H^0(X,L^p))$
such that
\begin{equation}
\TT_\nu^k(H_0)\longrightarrow H_p,\quad\text{as}\ k\to
+\infty\;.
\end{equation}
Then up to enlarging the constant
$C>0$ in \cref{expcvest}, we can assume that $H_0$
belongs to any fixed neighborhood $U\subset\Prod(H^0(X,L^p))$ of
$H_p$.
Consider $H_p$ as a base point metric as in \cref{basept},
so that any $H\in\Prod(H^0(X,L^p))$ is identified
with an Hermitian endomorphism $H\in\cL(\HH_p)$ via the
formula $H:=H_p(H\cdot,\cdot)$.
Take a neighborhood $U\subset\Prod(H^0(X,L^p))$ 
such that there is a diffeomorphism
\begin{equation}\label{diffU}
\begin{split}
U&\longrightarrow V\times I\\
H&\longmapsto\left(\frac{H}{\det(H)},\det(H)\right)\;,
\end{split}
\end{equation}
where $I\subset\R$ is a neighborhood of $1\in\R$ and
$V$ is a neighborhood of $H_p\simeq\Id_{\HH_p}$ in the space
of positive Hermitian endomorphisms of determinant $1$
acting on $\HH_p$.
In particular, the tangent space $T_{H_p}V$ is naturally
identified
with the space of
traceless endomorphisms in $\cL(\HH_p)$.
Then for any $H\in V$, the map
\begin{equation}\label{Tnu/det}
H\longmapsto\frac{\TT_\nu(H)}{\det(\TT_\nu(H))}
\end{equation}
fixes $H_p$, and its differential acts on traceless
Hermitian endomorphisms in $\cL(\HH_p)$ by
\begin{equation}\label{Tnulin}
D_{H_p}\TT_\nu-\Tr[D_{H_p}\TT_\nu]\,\Id_{\cL(\HH_p)}\;.
\end{equation}
By \cref{expcvlem}, it
is a self-adjoint operator with 
eigenvalues contained in $]0,1[\subset\R$, which
implies in particular
that the map \cref{Tnu/det} is a local diffeomorphism
around $H_p$ in $V$.
Furthermore, by the classical Hartman-Grobman theorem,
the map \cref{Tnu/det}
is conjugate by a local homeomorphism to its linearization at
$H_p$. In particular, taking $\beta_p\in\,]0,1[$ as the
largest eigenvalue of \cref{Tnu/det}, we get a constant $C>0$
such that for all $k\in\N$,
\begin{equation}\label{Tnudetexpcv}
\textup{dist}
\left(\frac{\TT_\nu^k(H_0)}{\det(\TT_\nu^k(H_0))},H_p
\right)\leq C\beta_p^k\;.
\end{equation}
In view of \cref{diffU}, we see that to get
the exponential convergence \cref{expcvest}
from \cref{Tnudetexpcv},
we need to show that there is a constant $C>0$ such that
for all $k\in\N$, we have
\begin{equation}\label{detcv}
\left|\det \TT^k_\nu(H_0)-1\right|<C\beta^k_p\;.
\end{equation}
To this end recall that the functional
$\Psi:\Prod(H^0(X,L^p))\fl\R$ of
\cref{BBGZ} is decreasing under iterations of $\TT_\nu$
and invariant with respect to the action
of $\R_+$ by multiplication. By \cref{Tnudetexpcv}
and the differentiability of $\Psi$, there exists a constant
$C>0$ such that for all $k\in\N$, we have
\begin{equation}
0\leq\Psi(\TT^k_\nu(H_0))-\Psi(H_p)\leq C\beta^r_p
\;.
\end{equation}
A both terms apearing in the definition \cref{Psidef} of $\Psi$
are decreasing by \cref{Lfctlem} and formula \cref{logconc}
respectively, we deduce in particular that for all $k\in\N$
big enough,
\begin{equation}
0\leq\log\det(\TT^k_\nu(H_0))\leq C\beta^k_p\;,
\end{equation}
from which \cref{detcv} follows. This completes the proof
of the exponential convergence \cref{expcvest}. The asymptotic
expansion \cref{betap} is then immediate consequence
of \cref{expcvlem}, and the fact that it is sharp
follows from the fact that \cref{Tnu/det}
is conjugate to its linearization
\cref{Tnulin} by a local homeomorphism.
\qed

\begin{rmk}\label{gaprmk}
Consider a general compact complex manifold $X$
equiped with an ample line bundle $L$, and
consider a volume map equal
to a constant value $d\nu\in\MM(X)$ as in \cref{cstantex}.
Then the asymptotics of the optimal rate of convergence
\cref{betap} for the associated Donaldson map
have been computed in \cite[Th.\,3.1, Rmk.\,4.12]{IKPS19},
and are given by the following estimate as $p\to+\infty$
\begin{equation}\label{betapcstant}
\beta_p=1-\frac{\lambda_1}
{4\pi p}+ O(p^{-2})\,,
\end{equation}
where $\lambda_1>0$ is the first eigenvalue of the
polarized Yau metric associated with $d\nu$.
Then if $X$ is a Fano manifold with $L:=K_X^*$
and if $d\nu\in\MM(X)$ is a Kähler-Einstein volume form
as in \cref{KEdef}, \cref{gapth} shows that the iterations
of the Donaldson map associated with
the constant volume map converge faster than
the iterations associated with
the anticanonical Donaldson map of \cref{Tdef}, as soon as
$p\in\N$ is
big enough. This behavior was predicted numerically
by Donaldson in
\cite[\S\,2.2.2]{Don09}.
Note that the iterations of the Donaldson map for the
constant volume map are of no practical
interest to approximate Kähler-Einstein metrics,
as one would need to know
the Kähler-Einstein volume form a priori.
By contrast, in case $X$ is a Calabi-Yau manifold,
the relevant
volume form $d\nu\in\MM(X)$
is purely determined by the complex geometry
of the manifold, and the iterations of the Donaldson map
in this case can be used to approximate numerically the
polarized Ricci-flat metric.

On the other hand, the methods of this paper also
apply to manifolds with $L:=K_X$ ample
and the canonical volume map of \cref{canex}, giving
the following estimate as $p\to+\infty$ for the
rate of convergence \cref{betap},
\begin{equation}\label{betapcstant2}
\beta_p=1-\frac{\lambda_1+4\pi}
{4\pi p}+ O(p^{-2})\,,
\end{equation}
where $\lambda_1>0$ is the first positive eigenvalue
of the Kähler-Einstein Laplacian acting on $\cinf(X,\C)$.
We then see that the iterations associated with
the canonical volume map
converge faster than both previous examples when $p\in\N$ is
large enough. Note that the existence of the Kähler-Einstein
metric in this case is the easiest case of the celebrated
theorem of Yau \cite{Yau78}, as shown by Aubin \cite{Aub78}.

Finally, using the methods of this paper and a refined
estimate on the spectral gap of the quantum channel,
one can show that
the rate of convergence of iterations for the Liouville
volume map of \cref{Liouvilleex} as $p\to+\infty$ satisfies
\begin{equation}\label{betapcstant3}
\beta_p=1+ O(p^{-2})\,,
\end{equation}
which also confirms a prediction of Donaldson
in \cite[\S\,2.2]{Don09}. \cref{gapth} thus shows that the convergence
of the iterations of Donaldson's map are much faster in the anticanonical
case than in the Liouville case, when $p\in\N$ is taken big enough.
\end{rmk}


\providecommand{\bysame}{\leavevmode\hbox to3em{\hrulefill}\thinspace}
\providecommand{\MR}{\relax\ifhmode\unskip\space\fi MR }
\providecommand{\MRhref}[2]{%
  \href{http://www.ams.org/mathscinet-getitem?mr=#1}{#2}
}
\providecommand{\href}[2]{#2}

\Addresses

\end{document}